\documentclass[12pt]{amsart}
\usepackage{graphicx}
\usepackage{tikz}        
\usepackage{amssymb}
\usepackage{epstopdf}
\usepackage[mathscr]{eucal}
\usepackage{amsmath,amssymb,amscd}
\usepackage{color}
\usepackage{epsfig}
\usepackage{bbold}

\newtheorem{lemma}{Lemma}

\newtheorem{definition}{Definition}

\newtheorem{proposition}{Proposition}
\theoremstyle{definition}

\def \mb{\mathbb}

\def \Z{\mb Z}                  
\def \R{\mb R}                 

\def \vp{\varphi}       

\newcommand {\csn} {\text{csn}}
\newcommand {\arccsn} {\text{arccsn}}
\newcommand {\sn} {\text{sn}}
\newcommand {\ctn} {\text{ctn}}

\newcommand {\csct} {\text{csct}}
\def \S{\mb S}        
\def \H{\mb H}        

\def\v{{\bf v}}

\newcommand {\q} {\mathbf{q}}

\newcommand {\F} {\mathbf{F}}

\def \and{\mbox{and}}
\usepackage[top=4cm, bottom=4cm, left=3.5cm, right=3.5cm]{geometry}
\title{The $N$-Body Problem in Spaces with Uniformly Varying Curvature}
\begin{document}
\maketitle
\markboth{Eric Boulter, Florin Diacu, and Shuqiang Zhu}{The $N$-body problem in spaces with uniformly varying curvature}
\author{\begin{center}
{\bf Eric Boulter}$^2$, {\bf Florin Diacu}$^{1,2}$, {\bf Shuqiang Zhu$^2$}\\
\bigskip
{\footnotesize $^1$Pacific Institute for the Mathematical Sciences\\
and\\
$^2$Department of Mathematics and Statistics\\
University of Victoria\\
P.O.~Box 1700 STN CSC\\
Victoria, BC, Canada, V8W 2Y2\\
\bigskip
boulter2@uvic.ca, diacu@uvic.ca, zhus@uvic.ca\\
}
\end{center}


\begin{abstract}
We generalize the curved $N$-body problem to spheres and hyperbolic spheres whose curvature $\kappa$ varies in time.
Unlike in the particular case when the curvature is constant, the equations of motion are non-autonomous. We first briefly consider the analogue of the Kepler problem and then investigate the homographic orbits for any number of bodies, proving the existence of several such classes of solutions on spheres. Allowing the curvature to vary in time offers some insight into the effect of an expanding universe, in the context the curved $N$-body problem, when $\kappa$ satisfies Hubble's law. The study of these equations also opens the possibility of finding new connections between classical mechanics and general relativity.
\end{abstract}



{
\tableofcontents

}


\section{Introduction}

In the 1830s, J\'anos Bolyai and Nikolai Lobachevsky independently thought that the laws of physics depend on the geometry of the universe, so they sought a natural extension of gravity to hyperbolic space, \cite{Bolyai}, \cite{Lobachevsky}. This idea led to the study of the Kepler problem and the 2-body problem in the framework of hyperbolic and elliptic geometry. Unlike in Euclidean space, the equations describing them are not equivalent, since the latter system is not integrable, \cite{Shchepetilov}. More recently, the problem was generalized to any number $N$ of bodies, leading to works such as \cite{Diacu}, \cite{Diacu2}, \cite {Diacu3}, \cite{Diacu5}, \cite{Diacu51}, \cite{DiacuKordlou}, \cite{Diacu6}, \cite{Diacu7}, \cite{DiacuPChavelaGRVictoria}, \cite{DiacuSanchez}, \cite{DiacuZhu}, \cite{DiacuThorn}, \cite{Garcia}, \cite{Martinez1}, \cite{PChavelaGRVictoria}, \cite{Shchepetilov}, \cite{Tibboel1}, \cite{Tibboel2}, \cite{Tibboel3}, and \cite{Zhu}. In the light of Hubble's law, \cite{Hubble}, a non-flat universe (i.e.\ a 3-sphere or a hyperbolic 3-sphere) would have uniformly varying curvature $\kappa=\kappa(t)$ as the universe expands, meaning that at a given time the curvature is the same at every point. Therefore by modifying the equations of the curved $N$-body problem to allow for uniformly varying curvature, we can construct a gravitational model that accounts for an expanding universe without requiring general relativity. Of course, we do not claim that the model we will introduce here could replace general relativity in cosmological studies. We are mostly interested in the mathematical aspects of a curved $N$-body problem on expanding or contracting spheres and hyperbolic spheres, a problem that, to our knowledge, has not been considered before in the framework of classical mechanics.

We are not only deriving here the equations of motion of this $N$-body problem, but will also examine how the uniformly varying curvature affects the system's behaviour and the existence of certain solutions. In Section 2 we find the equations of the $N$-body problem with uniformly varying curvature on the variable 3-sphere, 
$$
\S^3_\kappa:=\S^3_\kappa(t)=\{(x,y,z,w) \in \R^4:x^2+y^2+z^2+w^2=\kappa^{-1}(t),\ \kappa(t)>0\},
$$ 
and the variable hyperbolic 3-sphere, 
$$
\H^3_\kappa:=\H^3_\kappa(t)=\{(x,y,z,w) \in \R^{3,1}: x^2+y^2+z^2-w^2=\kappa^{-1}(t),\ \kappa(t)<0\},
$$ 
where $\R^{3,1}$ is the Minkowski space, by generalizing the derivation of the curved $N$-body equations with cotangent potential, as done in \cite{Diacu}, and perform a changeH of coordinates to reduce the problem to the study of the motion projected onto the unit manifolds $\S^3$ and $\H^3$, respectively. We then seek the first integrals of the equations and find a Lagrangian for the projected coordinates.  In Section 3, we derive the equations of the Kepler problem, a two-body system where one body is fixed, and rule out some of the solutions typically expected in such problems. In Section 4 we first define the concept of homographic solution, which consists of orbits for which the geometric configuration of the particles remains similar to itself during the motion while the curvature of the space changes in time. We then show that such orbits exist in $\S^3_\kappa$, but not in $\H^3_\kappa$, and that the homographic solutions of $\S^3_\kappa$ correspond to the special central configurations studied in \cite{DiacuZhu} for $\kappa$ constant, orbits for which the forces acting on the bodies cancel each other if the system is initially at rest. This observation allows us to reduce the existence of homographic solutions in a variable 3-sphere to a problem in the constant curvature case, which we can consider, without any loss of generality, in the unit sphere $\S^3$. In Section 5, we find several new special central configurations in $\S^3$, some of which lie completely on a great sphere $\S^2$. Among these orbits are those that satisfy certain necessary and sufficient conditions for the existence of a 4-body special central configuration as well as for a 5-body special central configuration, namely pentatopes in $\S^3$, which are not contained on any great 2-sphere.

We would like to mention that the idea of introducing and studying this problem came to us from Sergio Benenti's book in progress, \cite{Ben}, which was pointed out to the attention of Florin Diacu by his former doctoral student Manuele Santoprete. In his manuscript, Benenti develops a remarkable axiomatic setting for isotropic cosmological models, considering the spaces $\S_\kappa(t)$ and $\H_\kappa(t)$ as defined above. However, he shows no interest in deriving the equations of motion of an $N$-body problem, focusing instead on some deep cosmological questions he treats with relativistic techniques.   

\section{Equations of Motion}

In order to study the $N$-body problem in spaces with uniformly varying curvature, it is first necessary to generalize the equations of motion from the  constant curvature case by applying the Euler-Lagrange equations to the Lagrangian used in \cite{Diacu}, where $\kappa$ is a non-zero differentiable function of time. The goal of this section is to obtain the new system of equations and its basic integrals of motion.

\subsection{Deriving the equations of motion}

Let the curvature $\kappa:[0, \infty) \to \R$ be a non-zero differentiable function of time. Take $\q=(\q_1,\ldots, \q_N)$, with $\q_i \in \R^4$, if $\kappa(t) > 0$, but the Minkowski 4-space $\R^{3,1}$, if $\kappa(t) < 0$.
We define the potential energy to be $-U_\kappa$, where $U_k$ is the force function\begin{equation}\label{equ:potential}
U_\kappa(\q)= \sum\limits_{1 \leq i < j \leq N} \frac{m_i m_j |\kappa|^{1/2} \kappa \q_i \cdot \q_j} {[\sigma-\sigma(\kappa \q_i \cdot \q_j)^2]^{1/2}},\end{equation}
$\sigma$ denotes the sign of $\kappa$, and $\cdot$ is the standard inner product for $\kappa>0$, but the Lorentz product $\q_i \cdot \q_j=x_ix_j+y_iy_j+z_iz_j-w_iw_j$ for $\kappa<0$. When $\kappa$ is constant, $U_\kappa$ offers the natural extension of Newton's law to curved spaces, see \cite{Diacu}. 

We define the kinetic energy as
\begin{equation}T_\kappa(\dot\q)= \frac{1}{2} \sum\limits_{i=1}^N m_i \dot{\q}_i \cdot \dot{\q}_i,\end{equation}
so the Lagrangian function is $L_\kappa = T_\kappa + U_\kappa$. Consequently we can obtain in the standard manner the Euler-Lagrange equations with holonomic constraints, 
\begin{equation}
\frac{d}{dt} \frac{\partial L_\kappa}{\partial \dot{\q}_i} - \frac{\partial L_\kappa}{\partial \q_i}-\lambda_\kappa^i \frac{\partial f_\kappa^i}{\partial \q_i} = 0, \ i=1,\ldots, N,
\end{equation}
where $\kappa \neq 0$ and  $f_\kappa^i = \q_i \cdot \q_i -\frac{1}{\kappa} =0,\ i=1,\dots, N,$ are the constraints that keep the particle system on $\S^3_\kappa$ or $\H^3_\kappa$, respectively. The above system then becomes 
\begin{equation}\label{Lagrange} 
m_i \ddot{\q}_i = \nabla_{\q_i} U_\kappa+2\lambda_\kappa^i \q_i, \ i=1,\dots, N.
\end{equation}
Dot-multiplying these equations by $\q_i$ leads to 
\begin{equation}
\label{dotted} 
m_i\ddot{\q}_i \cdot \q_i = \nabla_{\q_i} U_\kappa \cdot \q_i +2\lambda_\kappa^i \q_i \cdot \q_i, \ i=1,\dots, N.
\end{equation} 
Since $U_\kappa$ is a homogeneous function of degree 0, it follows by Euler's formula for homogeneous functions that $\nabla_{\q_i} U \cdot \q_i = 0$. As $f_\kappa^i = 0$, we also have 
$$
\dot{f}_\kappa^i = 2\dot{\q}_i\cdot \q_i +\frac{\dot{\kappa}}{\kappa^2} = 0, \ i=1,\dots, N,
$$ 
and 
$$
\ddot{f}_\kappa^i = 2\ddot{\q}_i \cdot \q_i + 2\dot{\q}_i \cdot \dot{\q}_i +\frac{\ddot{\kappa}}{\kappa^2} - 2 \frac{\dot{\kappa}^2}{\kappa^3}=0, \ i=1,\dots, N.
$$ 
Substituting these into (\ref{dotted}) gives \begin{equation*} 
-m_i \dot{\q}_i \cdot \dot{\q}_i -\frac{m_i\ddot{\kappa}}{2\kappa^2} + \frac{m_i\dot{\kappa}^2}{\kappa^3}=2\frac{\lambda_\kappa^i}{\kappa}, \ i=1,\dots, N,
\end{equation*}
so 
\begin{equation}
\label{lambda}
\lambda_\kappa^i = -\frac{m_i \kappa}{2} \dot{\q}_i \cdot \dot{\q}_i - \frac{m_i\ddot{\kappa}}{4\kappa} + \frac{m_i\dot{\kappa}^2}{2\kappa^2}, \ i=1,\dots, N.
\end{equation}
If we substitute (\ref{lambda}) into (\ref{Lagrange}), we obtain that \begin{equation} \label{sizeVary} m_i \ddot{\q}_i = \nabla_{\q_i} U_\kappa - m_i \kappa (\dot{\q}_i \cdot \dot{\q}_i)\q_i - m_i\frac{\ddot{\kappa}}{2\kappa}\q_i + m_i\frac{\dot{\kappa}^2}{\kappa^2}\q_i, \ i=1,\dots, N,
\end{equation}
with the constraints
\begin{equation} 
\kappa\q_i \cdot \q_i = 1,\ i=1,\dots, N,\ \kappa \neq 0.
\end{equation}
The change of variables $\q_i = |\kappa|^{-1/2} \overline{\q}_i$ projects the system on $\S^3$ and $\H^3$. We obtain
$$
\dot{\q}_i = -\frac{\sigma \dot{\kappa} \overline{\q}_i}{2|\kappa|^{3/2}} + \frac{\dot{\overline{\q}}_i}{|\kappa|^{1/2}},
$$  
$$
\ddot{\q}_i= -\frac{\sigma \ddot{\kappa} \overline{\q}_i}{2|\kappa|^{3/2}} + \frac{3\dot{\kappa}^2 \overline{\q}_i}{4 |\kappa|^{5/2}} - \frac{\sigma \dot{\kappa} \dot{\overline{\q}}_i}{|\kappa|^{3/2}} + \frac{\ddot{\overline{\q}}_i}{|\kappa|^{1/2}},
$$
and the equations of motion take the form 
\begin{align*}
m_i\ddot{\overline{\q}}_i = &|\kappa|^{3/2} \nabla_{\overline{\q_i}} \overline{U} -m_i\left(\frac{\sigma\dot{\kappa}^2}{4\kappa^2} \overline{\q}_i \cdot \overline{\q}_i - \frac{\dot{\kappa}}{|\kappa|}\overline{\q}_i \cdot \dot{\overline{\q}}_i + \sigma \dot{\overline{\q}}_i \cdot \dot{\overline{\q}}_i \right)\overline{\q}_i\\
&+ \frac{m_i\dot{\kappa}^2}{4\kappa^2} \overline{\q}_i+\frac{m_i \dot{\kappa} }{\kappa}\dot{\overline{\q}}_i, \ i=1,\dots, N.
\end{align*}
We can also write this system as 
\begin{equation}\label{fixed}
m_i\ddot{\overline{\q}}_i = |\kappa|^{3/2} \nabla_{\overline{\q}_i} \overline{U} - \sigma m_i ( \dot{\overline{\q}}_i \cdot \dot{\overline{\q}}_i)\overline{\q}_i + \frac{m_i\dot{\kappa}}{\kappa} \dot{\overline{\q}}_i,\ i=1,\dots, N,
\end{equation}
with constraints
\begin{equation}
 \overline{\q}_i \cdot \overline{\q}_i = \sigma,\ \overline{\q}_i \cdot \dot{\overline{\q}}_i = 0, \ i=1,\dots, N,\ \kappa \neq 0.
\end{equation}
In \cite{Diacu3, DiacuZhu}, the explicit form of $\nabla_{\overline{\q}_i} \overline{U}$
is written as
\begin{equation} \label{equ:F}
\nabla_{\overline{\q}_i} \overline{U}=\sum\limits_{j=1, j\neq i}^N \frac{m_im_j(\overline{\q}_j-\sigma (\overline{\q}_i \cdot \overline{\q}_j)\overline{\q}_i)}{(1-(\overline{\q}_i \cdot\overline{ \q}_j)^2)^{3/2}}=\sum\limits_{j=1, j\neq i}^N \frac{m_im_j(\overline{\q}_j-\csn d_{ij}\overline{\q}_i)}{\sn^3 d_{ij}}.
\end{equation}
where $ \sn (x)= \sin(x) \ {\rm or}\ \sinh(x), \ 
   \ \csn (x)= \cos(x) \ {\rm or}\ \cosh(x), $ and 
 $d_{ij}$ is the distance between $\overline{\q}_i$ and $\overline{\q}_j$ in $\S^3$ or $\H^3$.  It is $
    d_{ij}:=\arccsn
    (\sigma \overline{\q}_i\cdot\overline{\q}_j). 
    $
For future reference, we introduce more notations which unify the  trigonometric and hyperbolic  functions \cite{Diacu3}, 
\begin{align*}
 \ctn (x)&= \csn(x)/\sn(x)= \cot(x) \ {\rm or}\ \coth(x), \\ 
   \ \csct (x)&= 1/\sn(x)=1/\sin(x) \ {\rm or}\ 1/\sinh(x).  \end{align*} 
We will also use $\F_i=\nabla_{\overline{\q}_i} \overline{U}$ to indicate  that this term can be  viewed as the attraction force on $\overline{\q}_i$ be all other particles.

\subsection{Integrals of the total angular momentum}

We can obtain the integrals of the total angular momentum. Consider the wedge product of $\overline{\q}_i$ and  the $i^{th}$ equations of (\ref{fixed}). For detail of the wedge product $(\wedge)$, the reader can see p.31 of \cite{Diacu3}. Dividing by $|\kappa|$, and summing the equations over $i$, we obtain
$$
\sum\limits_{i=1}^N \frac{m_i}{|\kappa|} \ddot{\overline{\q}}_i \wedge \overline{\q}_i = \sum\limits_{i=1}^N |\kappa|^{1/2} (\nabla_{\overline{\q}_i} \overline{U}) \wedge \overline{\q}_i -\sum\limits_{i=1}^N \left[\frac{m_i}{\kappa} (\dot{\overline{\q}}_i \cdot \dot{\overline{\q}}_i)\overline{\q}_i \wedge \overline{\q}_i + \frac{ m_i\dot{\kappa}}{\sigma\kappa^2} \dot{\overline{\q}}_i \wedge \overline{\q}_i\right].
$$
Since the wedge product, $\wedge$, is skew-symmetric, $\sum\limits_{i=1}^N |\kappa|^{1/2} (\nabla_{\overline{\q}_i} \overline{U}) \wedge \overline{\q}_i =0$. 
Combining this property with the fact that $\overline{\q}_i \wedge \overline{\q}_i = 0$, we obtain that
\begin{equation} 
\label{angmom}
\sum\limits_{i=1}^N\left[\frac{m_i}{|\kappa|} \ddot{\overline{\q}}_i \wedge \overline{\q}_i - \frac{\sigma m_i \dot{\kappa}}{\kappa^2} \dot{\overline{\q}}_i \wedge \overline{\q}_i\right]=0.
\end{equation}
This is the negative of the time derivative of the system's angular momentum about the origin, $\mathbf{L} = \sum\limits_{i=1}^N \frac{m_i}{|\kappa|} \overline{\q}_i \wedge \dot{\overline{\q}}_i$, and provides the six integrals
\begin{align*}
L_{wx}&=\sum\limits_{i=1}^N \frac{m_i}{|\kappa|}(y_i\dot{z}_i-z_i\dot{y}_i), 
&& L_{wy}=\sum\limits_{i=1}^N \frac{m_i}{|\kappa|}(x_i\dot{z}_i-z_i\dot{x}_i).\\
L_{wz}&=\sum\limits_{i=1}^N \frac{m_i}{|\kappa|}(x_i\dot{y}_i-y_i\dot{x}_i), 
&& L_{xy}=\sum\limits_{i=1}^N \frac{m_i}{|\kappa|}(w_i\dot{z}_i-z_i\dot{w}_i),\\
L_{xz}&=\sum\limits_{i=1}^N \frac{m_i}{|\kappa|}(w_i\dot{y}_i-y_i\dot{w}_i), 
&& L_{yz}=\sum\limits_{i=1}^N \frac{m_i}{|\kappa|}(w_i\dot{x}_i-x_i\dot{w}_i).
\end{align*}

\subsection{The $\overline{\q}$-Lagrangian}

If we have a Lagrangian in both the normal and projected coordinates, we may obtain the projected equations without having to first derive the full equations. The Lagrangian of the system in $\overline{\q}$ coordinates is
\begin{equation}\label{barLagrange}
\overline{L}=\sum\limits_{i=1}^N\frac{m_i(\dot{\overline{\q}}_i\cdot \dot{\overline{\q}}_i)}{2|\kappa|}+\sum\limits_{1\leq i<j\leq N}|\kappa|^{1/2}\frac{\sigma m_j m_i (\overline{\q}_j \cdot \overline{\q}_i)}{(\sigma-\sigma (\overline{\q}_j \cdot \overline{\q}_i)^2)^{1/2}}\end{equation}
with constraints $f_i=\overline{\q}_i \cdot \overline{\q}_i-\sigma=0$. Applying the Euler-Lagrange equations in terms of the $\overline{\q}_i$s gives system (\ref{fixed}).

\section{The Kepler Problem}

The simplest system that can be derived from the $N$-body problem is the Kepler problem, which describes the gravitational motion of a point mass $m$ about a fixed point mass $M$. Without loss of generality, we will assume $M$ is fixed at position $N=(0,0,0,1)$ in the $\overline{\q}$ coordinates. 

We will use the 3-(hyperbolic)-spherical coordinates. 

To unify the notation of trigonometric and hyperbolic functions, we introduce the following functions

If we convert to 3-(hyperbolic)-spherical coordinates $(\alpha, \theta, \varphi)$ by taking 
$$\overline{\q}= (x,y,z, w)=(\sn{\alpha}\sin{\theta}\cos{\varphi}, \sn{\alpha}\sin{\theta}\sin{\varphi}, \sn{\alpha}\cos{\theta}, \csn{\alpha}),$$
 equation (\ref{barLagrange}) becomes \begin{equation}
\overline{L}=\frac{m(\dot{\alpha}^2+\dot{\theta}^2\sn^2{\alpha}+\dot{\varphi}^2\sn^2{\alpha}\sin^2{\theta})}{2|\kappa|}+|\kappa|^{1/2}mM\ctn{\alpha}\end{equation}
with no constraint. Then the conjugate momenta for the system are 
$$
p_\alpha=\frac{m\dot{\alpha}}{|\kappa|},\
p_\theta=\frac{m\dot{\theta}\sn^2{\alpha}}{|\kappa|},\
p_\varphi=\frac{m\dot{\varphi}\sn^2{\alpha}\sin^2{\theta}}{|\kappa|},
$$
so the Hamiltonian has the form 
\begin{equation}
H=\frac{|\kappa|}{2m}(p_\alpha^2+p_\theta^2\csct^2{\alpha}+p_\varphi^2\csct^2{\alpha}\csc^2{\theta})-|\kappa|^{1/2}mM\ctn{\alpha},\end{equation}
and the equations of motion become
\begin{align}
\dot{\alpha}&=\frac{\partial H}{\partial p_\alpha}=\frac{|\kappa|p_\alpha}{m}\\
\dot{\theta}&=\frac{\partial H}{\partial p_\theta}=\frac{|\kappa|p_\theta\csct^2{\alpha}}{m}\\
\dot{\varphi}&=\frac{\partial H}{\partial p_\varphi}=\frac{|\kappa|p_\varphi\csct^2{\alpha}\csc^2{\theta}}{m}\\
\dot{p}_\alpha&=-\frac{\partial H}{\partial \alpha}=\frac{|\kappa|\ctn{\alpha}\csct^2{\alpha}}{m}(p_\theta^2+p_\varphi^2\csc^2{\theta})-|\kappa|^{1/2}mM\csct^2{\alpha}\\
\dot{p}_\theta&=-\frac{\partial H}{\partial \theta}=\frac{|\kappa|\csct^2{\alpha}\csc^2{\theta}\cot{\theta}p_\varphi^2}{m}\label{ptheta}\\
\dot{p}_\varphi&=-\frac{\partial H}{\partial \varphi}=0.\label{pphi}
\end{align}
From (\ref{pphi}) we have $A=p_\varphi$ is a constant, and direct computation leads to
\begin{align*}
L_{wz} &= p_\varphi\\
L_{wx} &= p_\theta \sin{\varphi}+p_\varphi \cot{\theta} \cos{\varphi}\\
L_{wy} &= p_\theta \cos{\varphi}+p_\varphi \cot{\theta} \sin{\varphi},
\end{align*} we have that their square sum $L=p_\theta^2+p_\varphi^2\csc^2{\theta}$ is a constant. Using this property, we can eliminate (\ref{ptheta}) and (\ref{pphi}) and obtain the equations of motion in the form
\begin{align}
\dot{\alpha}&=\frac{|\kappa|p_\alpha}{m}\label{alpha}\\
\dot{\theta}&=\pm\frac{|\kappa|\sqrt{(L-A^2\csc^2{\theta})}\csct^2{\alpha}}{m}\label{theta}\\
\dot{\varphi}&=\frac{|\kappa|A\csct^2{\alpha}\csc^2{\theta}}{m}\label{phi}\\
\dot{p}_\alpha&=\frac{|\kappa|L\ctn{\alpha}\csct^2{\alpha}}{m}-|\kappa|^{1/2}mM\csct^2{\alpha}.\label{palpha}
\end{align}
If we substitute (\ref{alpha}) into (\ref{palpha}) we get the uncoupled second order equation \begin{equation}
\ddot{\alpha}=\frac{\kappa^2L\ctn{\alpha}\csct^2{\alpha}}{m^2}-|\kappa|^{3/2}M\csct^2{\alpha}+\frac{\dot{\kappa}\dot{\alpha}}{\kappa}.\end{equation}

\subsection{Necessary condition on $\kappa$ for circular solutions}

The simplest solution of the Kepler problem in the Euclidean and constant curvature cases is the circular solution, i.e.\ an orbit for which the moving mass is at a constant distance from the fixed mass. We find that such solutions do not exist for non-constant curvature.

\begin{proposition}Circular orbits occur only in systems with constant curvature.\end{proposition}

\begin{proof}
Obviously, a circular orbit occurs when $\alpha$ has a constant value throughout the motion. By (\ref{alpha}), in order for this to be the case, we must have $p_\alpha=0$. But then, by (\ref{palpha}), $0=\frac{1}{m}|\kappa|L\ctn{\alpha}\csct^2{\alpha} - |\kappa|^{1/2}mM\csct^2{\alpha}$. If we isolate $\kappa$ we find that $$|\kappa|=\biggl(\frac{m^2M}{L\ctn{\alpha} }\biggr)^2.$$
Since the righthand  side consists only of constants, it follows that the system has circular orbits only if $\kappa$ does not depend on time.\end{proof}

We can also prove the following related result.
\begin{proposition}A system has non-fixed $T$-periodic solutions in phase space for some $T>0$ only if $\kappa$ is $T$-periodic.\end{proposition}
\smallskip

\begin{proof}
Let $(\alpha(t), \varphi(t), \theta(t), p_\alpha)$ be a solution to the curved Kepler problem with curvature $\kappa(t)$, such that for every $t \in [0, \infty)$, we have
\begin{align*}
\alpha(t+T)&=\alpha(t), && p_\alpha(t+T)=p_\alpha(t),\\
\theta(t+T)&=\theta(t), && \varphi(t+T)=\varphi(t)+2n\pi,
\end{align*}
for some $T \in \R, n \in \Z$. If $A\neq 0$, then by (\ref{phi}) $$0=\dot{\varphi}(t+T)-\dot{\varphi}(t)=\frac{(|\kappa(t+T)|-|\kappa(t)|) A\cdot\csct^2{\alpha(t)}\csc^2{\theta(t)}}{m},$$
so $|\kappa(t+T)|=|\kappa(t)|$. Since $\kappa$ is continuous and non-zero, there are no $t_1, t_2$ such that $\kappa(t_1)=-\kappa(t_2)$, so $\kappa(t+T)=\kappa(t)$.

If $A=0$, then by (\ref{theta}) $$0=\dot{\theta}(t+T)-\dot{\theta}(t)=\pm\frac{(|\kappa(t+T)|-|\kappa(t)|)\sqrt{L}\cdot\csct^2{\alpha(t)}}{m},$$
so by the same argument as above, $\kappa(t+T)=\kappa(t)$. Therefore $T$-periodic solutions occur only when $\kappa$ is $T$-periodic.
\end{proof}

\section{Homographic Orbits}
In this section, we study a class of rigid motions (rigid motions in the projected $\S^3$ or $\H^3$). We found that they exist in $\S^3$ but not in $\H^3$ and that they are related to special central configurations which was introduced in \cite{Diacu3, DiacuZhu}. 

\subsection{Homographic orbits in $\S^3$}
In $\S^3$, a solution in the form  $A^{-1}e^{\xi(t)}A$ is called a homographic orbits, where $A$ is a constant matrix in  $SO(4)$, and $$\xi(t)=\begin{bmatrix}0 & -\alpha(t) & 0 & 0\\ \alpha(t) & 0 & 0 & 0\\ 0 & 0 & 0 & -\beta(t)\\ 0 & 0 & \beta(t) & 0\end{bmatrix}, $$
$\alpha(t), \beta(t) \in C^1(\R),\ \alpha(0)= \beta(0)=0.$  Since the equations (\ref{fixed}) with $\kappa >0$ are invariant under the $SO(4)$-action, it is sufficient to consider the case $A=id_{SO(4)}$.

\begin{definition}[\cite{DiacuZhu}]
Consider the masses $m_1,\dots, m_N>0$ in $\S^3$. Then
a configuration
$$
\overline{\q}=(\overline{\q}_1, \overline{\q}_2 \ldots \overline{\q}_N),\ \overline{\q}_i=(x_i,y_i,z_i,w_i), \ i=1,...,N,
$$
is called a special central configuration if it is a critical point of the force function $\overline{U}$, i.e.
\begin{equation*}
\nabla_{\overline{\q}_i}\overline{U}(\overline{\q})=0,\ i=\overline{1,N}.
\end{equation*} 
\end{definition}

In $\S^3$, special central configurations leads to fixed-point solutions.
The next result shows that homographic orbits can be derived from the special central configurations in $\S^3$, meaning that finding homographic solutions on spheres with variable curvature is equivalent to finding fixed-point solutions in the unite sphere $\S^3$.

\begin{proposition}
Let $\overline{\q}(t)=(\overline{\q}_1(t), \overline{\q}_2(t) \ldots \overline{\q}_N(t))$ be a homographic motion of $\mathbb{S}^3_\kappa$. Then $\overline{\q}(t)$ is a solution to the $N$-body problem with time varying curvature if and only if $\overline{\q}$ is a special central configuration and $\overline{\q}_i(t)=e^{\xi(t)}\overline{\q}_i$ for $i=1,2,\ldots, N$, where \begin{equation} \xi(t)=\begin{bmatrix} 0 & -cK(t) & 0 & 0\\ cK(t) & 0 & 0 & 0\\ 0 & 0 & 0 & \pm cK(t)\\ 0 & 0 & \mp cK(t) & 0\end{bmatrix}, c \in \R, \label{xi}\end{equation}
with $K(t)=\int_0^t \kappa(\tau)d\tau$.
\end{proposition}
\begin{proof}
Let $\overline{\q}_i(t)=e^{\xi(t)}\overline{\q}_i$, where $\overline{\q}$ is a special central configuration, and $\xi(t)$ is defined as in (\ref{xi}). Then equations (\ref{fixed}) become 
\begin{equation*}
m_i(\ddot{\xi}(t)+\dot{\xi}^2(t))\overline{\q}_i=-m_i(\dot{\xi}(t)\overline{\q}_i \cdot \dot{\xi}(t)\overline{\q}_i)\overline{\q}_i+\frac{m_i\dot{\kappa}(t)}{\kappa(t)}\dot{\xi}(t)\overline{\q}_i,
\end{equation*}
since $\dot{\xi}(t)$ commutes with $e^{\xi(t)}$, and $\nabla_{\overline{\q}_i} U=0$. If we notice that 
$$
\dot{\xi}^2(t)\overline{\q}_i=-c^2\kappa^2(t)\overline{\q}_i,
\ \ \frac{\dot{\kappa}(t)}{\kappa(t)}\dot{\xi}(t)=\ddot{\xi}(t),
$$ 
then we have 
\begin{align*}
m_i\ddot{\xi}(t)\overline{\q}_i-m_ic^2\kappa^2(t)\overline{\q}_i&=-m_ic^2\kappa^2(x_i^2+y_i^2+z_i^2+w_i^2)\overline{\q}_i+m_i\ddot{\xi}(t)\overline{\q}_i\\
&=m_i\ddot{\xi}(t)\overline{\q}_i-m_ic^2\kappa^2(t)\overline{\q}_i,\end{align*}
so $\overline{\q}(t)$ is a solution to the $N$-body problem with time varying curvature when $\kappa >0$.
\smallskip

Conversely, suppose $\overline{\q}(t)=(\overline{\q}_1(t), \overline{\q}_2(t) \ldots \overline{\q}_N(t))$ is a solution to the $N$-body problem with uniformly varying positive curvature that is a homographic orbit in $\mathbb{S}^3_\kappa$. Then $$\overline{\q}_i(t)=\begin{bmatrix}x_i(t)\\ y_i(t) \\ z_i(t) \\ w_i(t)\end{bmatrix}=\begin{bmatrix}x_i \cos(\alpha(t))-y_i\sin(\alpha(t))\\ x_i\sin(\alpha(t))+y_i\cos(\alpha(t))\\ z_i\cos(\beta(t))-w_i\sin(\beta(t))\\ z_i\sin(\beta(t))+w_i\cos(\beta(t))\end{bmatrix},$$
where $\overline{\q}_i(0)=(x_i, y_i, z_i, w_i)^T$, and $\alpha,\beta$ are real differentiable functions such that $\alpha(0)=\beta(0)=0$. Notice that \begin{align*}
\dot{x}_i(t)&=-\dot{\alpha}(t)y_i(t),\\
\dot{y}_i(t)&=\dot{\alpha}(t)x_i(t),\\
\dot{z}_i(t)&=-\dot{\beta}(t)w_i(t),\\
\dot{w}_i(t)&=\dot{\beta}(t)z_i(t).
\end{align*}
If we look at the angular momentum integrals in the $xy$ and $zw$ directions, we find that 
\begin{align*}
L_{xy}&=\frac{1}{\kappa(t)}\sum\limits_{i=1}^N m_i(x_i(t)\dot{y}_i(t)-\dot{x}_i(t)y_i(t))\\
&=\frac{\dot{\alpha}(t)}{\kappa(t)}\sum\limits_{i=1}^Nm_i \bigl((x_i\cos(\alpha(t))-y_i\sin(\alpha(t)))^2+ (x_i\sin(\alpha(t))+y_i\cos(\alpha(t)))^2\bigr)\\
&=\frac{\dot{\alpha}(t)}{\kappa(t)}\sum\limits_{i=1}^N m_i(x_i^2+y_i^2),\\
\intertext{and}
L_{zw}&= \frac{1}{\kappa(t)}\sum\limits_{i=1}^Nm_i(x_i(t)\dot{y}_i(t)-\dot{x}_i(t)y_i(t))\\
&=\frac{\dot{\beta}(t)}{\kappa(t)}\sum\limits_{i=1}^N m_i\bigl((z_i\cos(\beta(t))-w_i\sin(\beta(t)))^2+(z_i\sin(\beta(t))+w_i\cos(\beta(t)))^2\bigr)\\
&=\frac{\dot{\beta}(t)}{\kappa(t)}\sum\limits_{i=1}^Nm_i(z_i^2+w_i^2).\end{align*}
For $L_{xy}$ and $L_{zw}$ to be constant, we must have one of the following three conditions satisfied:
\begin{enumerate}
\item $x_i=y_i=0$ for all $i=1,\ldots, N$ and $\beta(t)=cK(t)$ for some $c \in \R$;
\item $z_i=w_i=0$ for all $i=1,\ldots, N$ and $\alpha(t)=cK(t)$ for some $c \in \R$;
\item $\alpha(t)=aK(t)$ for some $a \in \R$ and $\beta(t)=bK(t)$ for some $b \in \R$.
\end{enumerate}
The first and second cases are proved the same way, so we will look at the first and third cases only.

In the first case, since $\overline{\q}_i(t)$ is a solution to the $N$-body problem with uniformly varying positive curvature, the following equation is satisfied for $i=1,\ldots, N$:
\begin{equation*}
m_ic\dot{\kappa}\begin{bmatrix}0\\ 0 \\ -w_i(t)\\ z_i(t)\end{bmatrix}-m_ic^2\kappa^2\overline{\q}_i(t)=\kappa^{3/2}(t)\nabla_{\overline{\q}_i} U - m_ic^2\kappa^2\overline{\q}_i(t) + m_ic\dot{\kappa}\begin{bmatrix}0 \\ 0 \\ -w_i(t) \\ z_i(t)\end{bmatrix}.\end{equation*}
Then $\nabla_{\overline{\q}_i} U=0$, so $\overline{\q}_i$ is a special central configuration, and 
\begin{align*}
\overline{\q}_i(t)&=\begin{bmatrix}0 \\ 0 \\  z_i\cos(cK)-w_i\sin(cK)\\ z_i\sin(cK)+w_i\cos(cK)\end{bmatrix}= e^{\xi(t)}\overline{\q}_i, 
\end{align*}
where $\xi(t)$ is as in (\ref{xi}).

In the third case, if we notice that $\frac{\dot{\kappa}(t)}{\kappa(t)}\dot{\xi}(t)=\ddot{\xi}(t)$, we know that the following equations are satisfied by $\overline{\q}_i(t)$:
\begin{align*}
\kappa^{3/2}\nabla_{\overline{\q}_i} U&=m_i\kappa^2(t)\begin{bmatrix}(a^2(x_i^2+y_i^2)+b^2(z_i^2+w_i^2)-a^2)x_i\\ (a^2(x_i^2+y_i^2)+b^2(z_i^2+w_i^2)-a^2)y_i\\ (a^2(x_i^2+y_i^2)+b^2(z_i^2+w_i^2)-b^2)z_i\\ (a^2(x_i^2+y_i^2)+b^2(z_i^2+w_i^2)-b^2)w_i\end{bmatrix}\\
&= m_i\kappa^2(t)(b^2-a^2)\begin{bmatrix}(z_i^2+w_i^2)x_i\\ (z_i^2+w_i^2)y_i\\ -(x_i^2+y_i^2)z_i\\ -(x_i^2+y_i^2)w_i\end{bmatrix}.\end{align*}
Assuming that $\kappa$ is not constant, this equation can only hold if $\overline{\q}_i(t)$ satisfies condition (1) or (2), or if $a=\pm b$ and $\overline{\q}_i$ is a special configuration. In either case, the hypothesis holds, and \begin{equation*}
\overline{\q}_i(t)=\begin{bmatrix}x_i\cos(aK)-y_i\sin(aK)\\ x_i\sin(aK)+y_i\cos(aK)\\ z_i\cos(aK)-w_i\sin(\pm aK)\\ z_i\sin(\pm aK)+w_i\cos(aK)\end{bmatrix}=e^{\xi(t)}\overline{\q}_i.\end{equation*}
This remark completes the proof.
\end{proof}
 
\subsection{Homographic orbits in $\H^3$}
In $\H^3$, a solution in the form  $B^{-1}e^{\xi_i(t)}B$,  where $i=1,2$, is called a homographic orbits, where $B$ is a constant matrix in  $SO(3,1)$, and 
$$\xi_1(t)=\begin{bmatrix}0 & 0 & 0 & 0\\ 0 & 0 & -\eta(t) & \eta(t)\\ 0 & \eta(t) & 0 & 0 \\ 0 & \eta(t) & 0 & 0\end{bmatrix}, \xi_2(t)=\begin{bmatrix}0 & -\alpha(t) & 0 & 0\\ \alpha(t) & 0 & 0 & 0\\ 0 & 0 & 0 & -\beta(t)\\ 0 & 0 & \beta(t) & 0\end{bmatrix},$$
$\alpha(t), \beta(t), \eta(t) \in C^1(\R),\ \alpha(0)= \beta(0)=\eta(0)=0.$ 
Since the equations (\ref{fixed}) with $\kappa<0$ are invariant under the $SO(3,1)$-action, it is sufficient to consider the case $B=id_{SO(3,1)}$.

As all homographic solutions in the case $\kappa>0$ correspond to fixed-point solutions, or special central configurations in $\S^3$, and there are no fixed-point solutions in $\H^3$ \cite{Diacu3, DiacuZhu}, we would expect that the case $\kappa<0$ does not have homographic solutions. We will further show that this is indeed the case.
\begin{proposition}
There are no homographic solutions of the $N$-body problem with negative uniformly varying curvature.
\end{proposition}
\begin{proof}
If a solution is  homographic, then it has the form 
$$
\overline{\q}_i(t)=e^{\xi_j(t)}\overline{\q}_i, \ j=1,2.
$$ We will rule out the two possible cases separately.
\smallskip

Case 1: $\xi=\xi_1$. In this case, solutions will take the form 
$$\overline{\q}_i(t)=\begin{bmatrix} x_i \\ y_i-z_i\eta(t)+w_i\eta(t)\\ z_i+y_i\eta(t)-z_i\eta^2(t)/2+w_i\eta^2(t)/2\\ w_i + y_i\eta(t)-z_i\eta^2(t)/2+w_i\eta^2(t)/2\end{bmatrix},$$
 where $\eta$ is a differentiable function, and $\overline{\q}_i=\begin{bmatrix}x_i\\y_i\\z_i\\w_i\end{bmatrix}$ is the initial position for the $i^{th}$ particle. If we look at the angular momentum integrals in the $xy, yz$ directions, we find after some simple calculations that \begin{align}
L_{xy}&=\frac{\dot{\eta}(t)}{\kappa(t)}\sum\limits_{i=1}^Nm_ix_i(w_i-z_i)\label{xy},\\
L_{yz}&=\frac{\dot{\eta}(t)}{\kappa(t)}\sum\limits_{i=1}^Nm_i(y_i^2+z_i^2-z_iw_i+\eta(t)y_i(w_i-z_i)+\frac{\eta^2(t)}{2}(w_i-z_i)^2).\label{yz}
\end{align}
Note that $\eta(t)$ is not constant. Otherwise, we get $\eta(t)=\eta(0)=0$, and we get a fixed-point solution in $\H^3$, which does not exist in $\H^3$\cite{Diacu3, DiacuZhu}. Thus either $\dot{\eta}(t)=c\kappa(t)$ for some $c \neq 0$, or \begin{align*}&\eta^2(t)\sum\limits_{i=1}^N\frac{m_i(w_i-z_i)^2}{2}+\eta(t)\sum\limits_{i=1}^N m_iy_i(w_i-z_i)+\sum\limits_{i=1}^Nm_i(y_i^2+z_i^2-z_iw_i)=0\\
\intertext{and}
&\sum\limits_{i=1}^Nm_ix_i(w_i-z_i)=0\end{align*}
for all $t \in [0, \infty)$. In the first case, in order for $L_{yz}$ to be constant, it would be necessary that $\sum_{i=1}^N\frac{m_i(w_i-z_i)^2}{2}=0$, so $w_i=z_i$ for all $i=1,2,\ldots N$. But if this is the case, then $x_i^2+y_i^2+z_i^2-w_i^2=x_i^2+y_i^2=-1$, which is impossible. In the second case, each of $\sum_{i=1}^N m_ix_i(w_i-z_i)$, $\sum_{i=1}^N\frac{m_i(w_i-z_i)^2}{2}$, $\sum_{i=1}^N m_iy_i(w_i-z_i)$, and $\sum_{i=1}^N m_i(y_i^2+z_i^2-z_iw_i)$ must be equal to zero. This is possible only if $y_i=0, z_i=w_i$ for $i=1,\ldots, N$. But then $x_i^2+y_i^2+z_i^2-w_i^2=x_i^2=-1$, which is impossible. Therefore there are no homographic orbits for $\xi_1$.
\smallskip

Case 2:  $\xi=\xi_2$. In this case, solutions will take the form $$\overline{\q}_i(t)=\begin{bmatrix}x_i\cos(\alpha(t)-y_i\sin(\alpha(t) \\ x_i \sin(\alpha(t))+y_i\cos(\alpha(t))\\ z_i\cosh(\beta(t))+w_i\sinh(\beta(t))\\ z_i\sinh(\beta(t))+w_i\cosh(\beta(t))\end{bmatrix},$$ where $\alpha, \beta$ are real-valued differentiable functions such that $\alpha(0)=\beta(0)=0$, and $\overline{\q}_i=\begin{bmatrix}x_i\\y_i\\z_i\\w_i\end{bmatrix}$ is the initial position for the $i^{th}$ particle. Notice that \begin{align*}
\dot{x}_i(t)&=-\dot{\alpha}(t)y_i(t),\\
\dot{y}_i(t)&=\dot{\alpha}(t)x_i(t),\\
\dot{z}_i(t)&=\dot{\beta}(t)w_i(t),\\
\dot{w}_i(t)&=\dot{\beta}(t)z_i(t).
\end{align*}
If we look at the angular momentum integrals in the $xy$ and $zw$ directions, we find that 
$$
L_{xy}=\frac{1}{\kappa(t)}\sum\limits_{i=1}^N m_i(x_i(t)\dot{y}_i(t)-\dot{x}_i(t)y_i(t))
$$
$$
=\frac{\dot{\alpha}(t)}{\kappa(t)}\sum\limits_{i=1}^Nm_i \bigl((x_i\cos(\alpha(t))-y_i\sin(\alpha(t)))^2+ (x_i\sin(\alpha(t))+y_i\cos(\alpha(t)))^2\bigr)
$$
$$
=\frac{\dot{\alpha}(t)}{\kappa(t)}\sum\limits_{i=1}^N m_i(x_i^2+y_i^2),
$$
$$
L_{zw}= \frac{1}{\kappa(t)}\sum\limits_{i=1}^N m_i(z_i(t)\dot{w}_i(t)-\dot{z}_i(t)w_i(t))
$$
$$
= \frac{\dot{\beta}(t)}{\kappa(t)}\sum\limits_{i=1}^N m_i((z_i\cosh(\beta(t))+w_i\sinh(\beta(t)))^2-(z_i\sinh(\beta(t))+w_i\cosh(\beta(t)))^2)
$$
$$
= \frac{\dot{\beta}(t)}{\kappa(t)}\sum\limits_{i=1}^N m_i(z_i^2-w_i^2).
$$
Since $z_i^2-w_i^2$ is always negative, $L_{zw}$ is constant only if $\beta(t)=bK(t)$ for some $b \in \R$. $L_{xy}$ is constant if either $\alpha(t)=aK(t)$ for some $a \in \R$ or $x_i=y_i=0$ for all $i=1,\ldots, N$. If $x_i=y_i=0$ for all $i=1,\ldots, N$, then the system satisfies the equation \begin{equation*}
m_ib\dot{\kappa}\begin{bmatrix}0\\0\\w_i\\z_i\end{bmatrix}+m_ib^2\kappa^2\overline{\q}_i = \kappa^{3/2}\nabla_{\overline{\q}_i} U + m_ib^2\kappa^2\overline{\q}_i + m_ib\dot{\kappa}\begin{bmatrix}0\\0\\w_i\\z_i\end{bmatrix}.
\end{equation*}
Consequently $\nabla_{\overline{\q}_i} U=0$, which is impossible for $\kappa <0$\cite{Diacu3, DiacuZhu}. If $\alpha(t)=aK(t)$, we notice that $\frac{\dot{\kappa}(t)}{\kappa(t)}\dot{\xi}(t)=\ddot{\xi}(t)$, so $\overline{\q}_i(t)$ satisfies the following equations:
\begin{align*}
\kappa^{3/2}(t) \nabla_{\overline{\q}_i} U&= m_i\kappa^2(t) \begin{bmatrix} (b^2(z_i^2-w_i^2)-a^2(x_i^2+y_i^2)-a^2)x_i\\ (b^2(z_i^2-w_i^2)-a^2(x_i^2+y_i^2)-a^2)y_i\\ (b^2+b^2(z_i^2-w_i^2)-a^2(x_i^2+y_i^2))z_i\\ (b^2+b^2(z_i^2-w_i^2)-a^2(x_i^2+y_i^2))w_i\end{bmatrix}\\
&= m_i(a^2+b^2)\kappa^2(t)\begin{bmatrix}(z_i^2-w_i^2)x_i\\ (z_i^2-w_i^2)y_i\\ -(x_i^2+y_i^2)z_i\\ -(x_i^2+y_i^2)w_i\end{bmatrix}.\end{align*}
Assuming that $\kappa$ is not constant, this equation can only hold if $\nabla_{\overline{\q}_i} U=0$, which is impossible for $\kappa<0$. Therefore there are no homographic orbits for $\xi_2$.
\end{proof}

\section{Special Central Configurations}

Since we have shown in the previous section that there is a strong link between homographic orbits and special central configurations in $\S^3$, we will now look at several examples of special central configurations and provide a rough classification of all 4-body special central configurations. In this section we will make use of several results about special central configurations that have been proved in \cite{DiacuZhu}:
\begin{enumerate}
\item No special central configuration in $\S^3$ has all masses lying in any closed hemisphere, unless all masses lie on a great sphere.
\item No special central configuration in $\S^2$ has all masses lying in any closed hemisphere, unless all masses lie on a great circle.
\item If $\q$ is a special central configuration in $\S^3$, and $g \in SO(4)$, then the configuration $g\q$, resulting from the action of $g$ on $\q$, is also a special central configuration.
\end{enumerate}
\smallskip

\subsection{Double Lagrangian special central configurations on $\S^2_{xyz}$}
Let $\S^2_{xyz}:=\{(x,y,z,w)\in \R^4| x^2+y^2+z^2=1, w=0\}.$   One of the simplest central configurations is the Lagrangian, consisting of 3 bodies of equal masses evenly spaced around a circle\cite{Diacu3, DiacuZhu}. We now look at the special central configurations consisting of two parallel Lagrangian central configurations, which we will call double Lagrangian.

\begin{proposition}In the 6-body problem on the sphere, there are infinitely many double Lagrangian special central configurations, i.e.\ configurations of the form 
\begin{align*}
&\q_1=\begin{bmatrix}r_1\\0\\c_1\\0\end{bmatrix}, \q_2=\begin{bmatrix}-\frac{r_1}{2}\\ \frac{\sqrt{3}r_1}{2}\\ c_1\\0\end{bmatrix}, \q_3=\begin{bmatrix}-\frac{r_1}{2}\\-\frac{\sqrt{3}r_1}{2} \\ c_1\\0\end{bmatrix},\\ &\q_4=\begin{bmatrix}r_2\\0\\c_2\\0\end{bmatrix}, \q_5=\begin{bmatrix}-\frac{r_2}{2}\\ \frac{\sqrt{3}r_2}{2}\\ c_2\\0\end{bmatrix}, \q_6=\begin{bmatrix}-\frac{r_2}{2}\\ -\frac{\sqrt{3}r_2}{2}\\c_2\\0\end{bmatrix},\\
&m_1=m_2=m_3=1,\\
&m_4=m_5=m_6=m,
\end{align*}
where $c_1 \in (0,1)$, $c_2 \in (-1,0)$, $r_1=\sqrt{1-c_1^2}$, $r_2=\sqrt{1-c_2^2}$ and $m \in (0,\infty)$.
\end{proposition}
\begin{proof}
To obtain a special central configuration in the 6-body problem on the sphere, we must have \begin{equation*}
\nabla_{\q_i} U=\sum\limits_{j=1, j\neq i}^6 \frac{m_im_j(\q_j-(\q_i \cdot \q_j)\q_i)}{(1-(\q_i \cdot \q_j)^2)^{3/2}}=0\end{equation*} for $i=1,2,3,4,5,6$. By symmetry arguments, it is sufficient for the equations to hold for $i=1$ and $4$. It is easy to see that the $y$ and $w$ components are 0 for all $c_1, c_2, m$. We can also see that 
$$
\q_i \cdot \nabla_{\q_i} U=\sum\limits_{j=1, j\neq i}^N m_im_j\frac{\q_j \cdot \q_i-\q_j \cdot \q_i}{(1-(\q_i \cdot \q_j)^2)^{3/2}}=0,
$$
so $\nabla_{\q_i} U$ is orthogonal to $\q_i$. Therefore the $z$ component of $\nabla_{\q_i} U$ is zero if and only if the $x$ component is zero, so it is sufficient to have the following two equations satisfied:
\begin{align}
\label{q1}0&=\frac{3r_1^2c_1}{(1-(c_1^2-\frac{r_1^2}{2})^2)^{3/2}}+\frac{m(c_2-(c_1c_2+r_1r_2)c_1)}{(1-(c_1c_2+r_2r_1)^2)^{3/2}}+\frac{m(2c_2-(2c_1c_2-r_1r_2)c_1)}{(1-(c_1c_2-\frac{r_1r_2}{2})^2)^{3/2}}\\
\label{q4}0&=\frac{c_1-(c_1c_2+r_1r_2)c_2}{(1-(c_1c_2+r_2r_1)^2)^{3/2}}+\frac{2c_1-(2c_1c_2-r_1r_2)c_2}{(1-(c_1c_2-\frac{r_1r_2}{2})^2)^{3/2}}+\frac{3mr_2^2c_2}{(1-(c_2^2-\frac{r_2^2}{2})^2)^{3/2}}.
\end{align}
By isolating $m$ in (\ref{q4}) we get \begin{equation}
\label{m}m=-\frac{(1-(c_2^2-\frac{r_2^2}{2})^2)^{3/2}}{3r_2^2c_2}\biggl(\frac{c_1-(c_1c_2+r_1r_2)c_2}{(1-(c_1c_2+r_2r_1)^2)^{3/2}}+\frac{2c_1-(2c_1c_2-r_1r_2)c_2}{(1-(c_1c_2-\frac{r_1r_2}{2})^2)^{3/2}}\biggr).
\end{equation}
If $(c_1,c_2,m)$ satisfy the requirements for a special configuration, then by symmetry so do $(-c_2,-c_1,\frac{1}{m})$, so we can find all special central configurations with $c_1 \geq -c_2$ and then obtain the equivalent special central configurations with $c_1<-c_2$. Let 
$$
B=\{(a,b) \in (0,1) \times (-1,0): a\geq -b\}.
$$ 
If we consider the function 
$$
f\colon B \to \R,
$$  
\begin{align*}
f(c_1,c_2)=&\frac{3r_1^2c_1}{(1-(c_1^2-\frac{r_1^2}{2})^2)^{3/2}}+\frac{m(c_2-(c_1c_2+r_1r_2)c_1)}{(1-(c_1c_2+r_2r_1)^2)^{3/2}}\\
&+\frac{m(2c_2-(2c_1c_2-r_1r_2)c_1)}{(1-(c_1c_2-\frac{r_1r_2}{2})^2)^{3/2}},\end{align*}
where $m$ is as in (\ref{m}), then since $B$ is path-connected, there exists a path 
$$
p:[0,1] \to B
$$
such that $p(0)=(\frac{1}{10},-\frac{1}{10})$ and $p(1)=(\frac{9}{10},-\frac{1}{2})$. The function $f$ is continuous on its domain, so $f \circ p$ is continuous on $[0,1]$, and since $f(p(0))<0$ and $f(p(1))>0$, we have by the intermediate value theorem that there exists $a \in [0,1]$ such that $f(p(a))=0$. Since this is true for any such path $p$, then if
$$
A=\{(c_1,c_2) \in B: f(c_1,c_2)=0\},$$
the set $B\setminus A$ must have $(\frac{9}{10}, -\frac{1}{2})$ in a different path component than $(\frac{1}{10}, -\frac{1}{10})$. No finite set can path-disconnect $B$, so $f(c_1,c_2)=0$ has infinitely many solutions. But $(c_1,c_2)$ is a special central configuration if $(c_1,c_2) \in A$, and $m(c_1,c_2) >0$. If $(c_1,c_2) \in B$, then $m(c_1,c_2)>0$ if and only if \begin{equation*}
\frac{c_1-(c_1c_2+r_1r_2)c_2}{(1-(c_1c_2+r_2r_1)^2)^{3/2}}+\frac{2c_1-(2c_1c_2-r_1r_2)c_2}{(1-(c_1c_2-\frac{r_1r_2}{2})^2)^{3/2}}>0,
\end{equation*}
$|c_1c_2+r_1r_2| < 1$  since $c_1c_2+r_1r_2=\cos(d_{14})$, and $|2c_1c_2-r_1r_2|< 2$ since $c_1c_2-\frac{r_1r_2}{2}=\cos(d_{15})$, so \begin{align*}
&\frac{c_1-(c_1c_2+r_1r_2)c_2}{(1-(c_1c_2+r_2r_1)^2)^{3/2}}+\frac{2c_1-(2c_1c_2-r_1r_2)c_2}{(1-(c_1c_2-\frac{r_1r_2}{2})^2)^{3/2}}\\
> &\frac{c_1+c_2}{(1-(c_1c_2+r_1r_2)^2)^{3/2}}+\frac{2c_1+2c_2}{(1-(c_1c_2-\frac{r_1r_2}{2})^2)^{3/2}}
\geq 0\end{align*}
since $c_1 \geq -c_2 $ for $(c_1,c_2) \in B$. Then $m$ is always positive in $B$, so every element $(c_1,c_2) \in A$ corresponds to the special central configuration 
\begin{align*}
&\q_1=\begin{bmatrix}\sqrt{1-c_1^2}\\0\\c_1\\0\end{bmatrix}, \q_2=\begin{bmatrix}-\frac{\sqrt{1-c_1^2}}{2}\\ \frac{\sqrt{3}\sqrt{1-c_1^2}}{2}\\ c_1\\0\end{bmatrix}, \q_3=\begin{bmatrix}-\frac{\sqrt{1-c_1^2}}{2}\\-\frac{\sqrt{3}\sqrt{1-c_1^2}}{2} \\ c_1\\0\end{bmatrix},\\
&\q_4=\begin{bmatrix}\sqrt{1-c_2^2}\\0\\c_2\\0\end{bmatrix}, \q_5=\begin{bmatrix}-\frac{\sqrt{1-c_2^2}}{2}\\ \frac{\sqrt{3}\sqrt{1-c_2^2}}{2}\\ c_2\\0\end{bmatrix}, \q_6=\begin{bmatrix}-\frac{\sqrt{1-c_2^2}}{2}\\ -\frac{\sqrt{3}\sqrt{1-c_2^2}}{2}\\c_2\\0\end{bmatrix},\\
&m_1=m_2=m_3=1,\\
&m_4=m_5=m_6=m(c_1,c_2),\\
\end{align*}
where $m(c_1,c_2)$ is as defined in (\ref{m}). This remark completes the proof.
\end{proof}

To get a visual understanding of the roots of $f$, we substitute (\ref{m}) into (\ref{q1}) and implicitly plot the solutions of the resulting equation. We can then see the set of solutions to $f(c_1,c_2)=0$, 
where the curves are solutions, and the  shaded region is formed by the $(c_1, c_2)$ values for which $m(c_1,c_2)\leq 0$. Since no solution occurs in the  shaded region, all these solutions represent special central configurations. As we showed above, the right branch of the solution set is a path-disconnecting subset of $B$, the solutions are symmetric about $c_1=-c_2$, and $m$ is positive on $B$.

\begin{figure}[htbp] 
   \centering
   \includegraphics[width=2in]{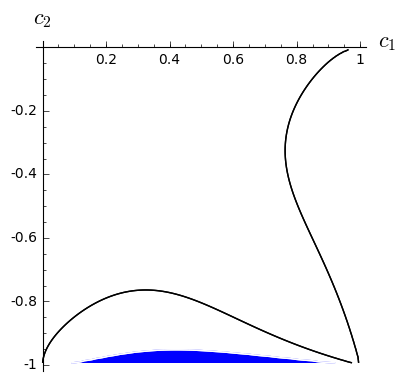}
    \caption{}
   \label{Fig2}
\end{figure}

\smallskip

\subsection{Double tetrahedron special central configurations in $\S^3$}
We now extend the previous case from two triangles in $\S^2_{xyz}$ to two tetrahedra in $\S^3$. We will call such a solution of the 8-body problem of the sphere a
double tetrahedron special central configuration.

\begin{proposition}In the 8-body problem in $\S^3$, there exist infinitely many double tetrahedron special central configurations, i.e.\ configurations of the form 
\begin{align*}
&\q_1=\begin{bmatrix}r_1 \\ 0 \\ 0 \\ c_1\end{bmatrix}, \q_2=\begin{bmatrix}-\frac{r_1}{3} \\ \frac{2\sqrt{2}r_1}{3} \\ 0 \\ c_1\end{bmatrix}, \q_3=\begin{bmatrix}-\frac{r_1}{3}\\-\frac{\sqrt{2}r_1}{3}\\ \frac{\sqrt{6}r_1}{3}\\c_1\end{bmatrix}, \q_4=\begin{bmatrix} -\frac{r_1}{3}\\-\frac{\sqrt{2}r_1}{3}\\-\frac{\sqrt{6}r_1}{3}\\c_1\end{bmatrix},\\
&\q_5=\begin{bmatrix}r_2 \\ 0 \\ 0 \\ c_2\end{bmatrix}, \q_6=\begin{bmatrix}-\frac{r_2}{3} \\ \frac{2\sqrt{2}r_2}{3} \\ 0 \\ c_2\end{bmatrix}, \q_7=\begin{bmatrix}-\frac{r_2}{3}\\-\frac{\sqrt{2}r_2}{3}\\ \frac{\sqrt{6}r_2}{3}\\c_2\end{bmatrix}, \q_8=\begin{bmatrix} -\frac{r_2}{3}\\-\frac{\sqrt{2}r_2}{3}\\-\frac{\sqrt{6}r_2}{3}\\c_2\end{bmatrix},\\
&m_1=m_2=m_3=m_4=1,\\
&m_5=m_6=m_7=m_8=m,\end{align*}
where $c_1 \in (0,1), c_2 \in (-1,0), m \in (0,\infty), r_1=\sqrt{1-c_1^2}$, and $r_2=\sqrt{1-c_2^2}$.
\end{proposition}
\begin{proof}
In order to have a special central configuration in the 8-body problem on the sphere, we must have \begin{equation}\label{sccreq}
\nabla_{\q_i} U=\sum\limits_{j=1, j\neq i}^8 \frac{m_im_j(\q_j-(\q_i \cdot \q_j)\q_i)}{(1-(\q_i \cdot \q_j)^2)^{3/2}}=0\end{equation} for $i=1,2,3,4,5,6, 7, 8$. Let 
\begin{equation*}
g=\begin{bmatrix}-\frac{1}{3} & -\frac{\sqrt{2}}{3} & \frac{\sqrt{6}}{3} & 0\\ \frac{2\sqrt{2}}{3} & -\frac{1}{6} & \frac{\sqrt{3}}{6} & 0\\ 0 & \frac{\sqrt{3}}{2} & \frac{1}{2} & 0\\ 0 & 0 & 0 & 1\end{bmatrix} \in \text{SO}(4),\ \ h=\begin{bmatrix}1 & 0 & 0 & 0\\ 0 & -\frac{1}{2} & -\frac{\sqrt{3}}{2} & 0\\ 0 & \frac{\sqrt{3}}{2} & -\frac{1}{2} & 0\\ 0 & 0 & 0 & 1\end{bmatrix} \in \text{SO}(4).\end{equation*}
The action of $\langle g, h \rangle$ on $\q$ is the permutation group $$\langle (\q_1, \q_2, \q_3)(\q_5, \q_6, \q_7), (\q_2, \q_3, \q_4)(\q_6, \q_7, \q_8)\rangle,$$ so by the symmetries of $\langle g, h \rangle$ it is sufficient for (\ref{sccreq}) to hold for $i=1,5$. For these two vertices, (\ref{sccreq}) becomes \begin{align*}
\nabla_{\q_1}U=&\frac{\q_2+\q_3+\q_4-3(c_1^2-\frac{r_1^2}{3})\q_1}{(1-(c_1^2-\frac{r_1^2}{3})^2)^{3/2}}+m\frac{\q_5-(c_1c_2+r_1r_2)\q_1}{(1-(c_1c_2+r_1r_2)^2)^{3/2}},\\
&+m\frac{\q_6+\q_7+\q_8-3(c_1c_2-\frac{r_1r_2}{3})\q_1}{(1-(c_1c_2-\frac{r_1r_2}{3})^2)^{3/2}}\\
\nabla_{\q_5}U=&m\frac{\q_2+\q_3+\q_4-3(c_1c_2-\frac{r_1r_2}{3})\q_5}{(1-(c_1c_2-\frac{r_1r_2}{3})^2)^{3/2}}+m\frac{\q_1-(c_1c_2+r_1r_2)\q_5}{(1-(c_1c_2+r_1r_2)^2)^{3/2}}\\
&+m^2\frac{\q_6+\q_7+\q_8-3(c_2^2-\frac{r_2^2}{3})\q_5}{(1-(c_2^2-\frac{r_2^2}{3})^2)^{3/2}}.
\end{align*}
Then since 
$$
\q_2+\q_3+\q_4=\begin{bmatrix}-r_1 \\ 0 \\ 0 \\ 3c_1\end{bmatrix}\ \text{and }\ \q_6+\q_7+\q_8=\begin{bmatrix}-r_2 \\ 0 \\ 0 \\ 3c_2\end{bmatrix},
$$ 
we can see that the $y$ and $z$ coordinates of $\nabla_{\q_1}U,\nabla_{\q_5}U$ are identically 0.
As we showed in the previous theorem, $\q_i \cdot \nabla_{\q_i}U=0$, so the $w$ component of $\nabla_{\q_i}U$ is 0 if and only if the $x$ component is 0 for $i=1,5$. Then the system must satisfy the equations:
\begin{align}
\label{q1tet}0&=\frac{4r_1^2c_1}{(1-(c_1^2-\frac{r_1^2}{3})^2)^{3/2}}+\frac{m(c_2-(c_1c_2+r_1r_2)c_1)}{(1-(c_1c_2+r_2r_1)^2)^{3/2}}+\frac{m(3c_2-(3c_1c_2-r_1r_2)c_1)}{(1-(c_1c_2-\frac{r_1r_2}{3})^2)^{3/2}}\\
\label{q5tet}0&=\frac{c_1-(c_1c_2+r_1r_2)c_2}{(1-(c_1c_2+r_2r_1)^2)^{3/2}}+\frac{3c_1-(3c_1c_2-r_1r_2)c_2}{(1-(c_1c_2-\frac{r_1r_2}{3})^2)^{3/2}}+\frac{4mr_2^2c_2}{(1-(c_2^2-\frac{r_2^2}{4})^2)^{3/2}}.\end{align}
By isolating $m$ in (\ref{q5tet}), we obtain\begin{equation}
\label{mtet}m=-\frac{(1-(c_2^2-\frac{r_2^2}{3})^2)^{3/2}}{4r_2^2c_2}\biggl(\frac{c_1-(c_1c_2+r_1r_2)c_2}{(1-(c_1c_2+r_2r_1)^2)^{3/2}}+\frac{3c_1-(3c_1c_2-r_1r_2)c_2}{(1-(c_1c_2-\frac{r_1r_2}{3})^2)^{3/2}}\biggr).
\end{equation}
If the elements $(c_1,c_2,m)$ satisfy the requirements for a special configuration, then by the symmetry under $\begin{bmatrix}1 & 0 & 0 & 0\\ 0 & 1 & 0 & 0\\ 0 & 0 & -1 & 0\\ 0 & 0 & 0 & -1\end{bmatrix}$, so do the elements $(-c_2,-c_1,\frac{1}{m})$, therefore we can find all special central configurations with $c_1 \geq -c_2$, and then obtain the equivalent special central configurations with $c_1<-c_2$. Consider the set 
$$
B=\{(a,b) \in (0,1) \times (-1,0): a\geq -b\}.
$$ 
If we take the function 
$$
g\colon B \to \R,
$$ 
\begin{align*}
g(c_1,c_2)=&\frac{4r_1^2c_1}{(1-(c_1^2-\frac{r_1^2}{3})^2)^{3/2}}+\frac{m(c_2-(c_1c_2+r_1r_2)c_1)}{(1-(c_1c_2+r_2r_1)^2)^{3/2}}\\
&+\frac{m(3c_2-(3c_1c_2-r_1r_2)c_1)}{(1-(c_1c_2-\frac{r_1r_2}{3})^2)^{3/2}},\end{align*}
where $m$ is as in (\ref{mtet}), then since $B$ is path-connected, there exists a path 
$$
p\colon[0,1] \to B
$$ 
such that $p(0)=(\frac{1}{10},-\frac{1}{10})$ and $p(1)=(\frac{9}{10},-\frac{1}{2})$. The function $g$ is continuous on its domain, so $g \circ p$ is a continuous function on $[0,1]$, and since $g(p(0))<0$ and $g(p(1))>0$, we have by the intermediate value theorem that there exists $a \in [0,1]$ such that $g(p(a))=0$. Since this is true for any such path $p$, then if 
$$
C=\{(c_1,c_2) \in B: g(c_1,c_2)=0\},
$$
the set $B\setminus C$ must have $(\frac{9}{10}, -\frac{1}{2})$ in a different path component than $(\frac{1}{10}, -\frac{1}{10})$. No finite set can path-disconnect $B$, so $g(c_1,c_2)=0$ has infinitely many solutions. But $(c_1,c_2)$ is a special central configuration if $(c_1,c_2) \in C$, and $m(c_1,c_2) >0$. If $(c_1,c_2) \in B$, so $m(c_1,c_2)>0$ if and only if \begin{equation*}
\frac{c_1-(c_1c_2+r_1r_2)c_2}{(1-(c_1c_2+r_2r_1)^2)^{3/2}}+\frac{3c_1-(3c_1c_2-r_1r_2)c_2}{(1-(c_1c_2-\frac{r_1r_2}{3})^2)^{3/2}}>0,
\end{equation*}
$|c_1c_2+r_1r_2| < 1$, since $c_1c_2+r_1r_2=\cos(d_{15})$, and $|3c_1c_2-r_1r_2|< 3$, since 
$$
c_1c_2-\frac{r_1r_2}{3}=\cos(d_{16}),
$$ so 
\begin{align*}
&\frac{c_1-(c_1c_2+r_1r_2)c_2}{(1-(c_1c_2+r_2r_1)^2)^{3/2}}+\frac{3c_1-(3c_1c_2-r_1r_2)c_2}{(1-(c_1c_2-\frac{r_1r_2}{3})^2)^{3/2}}\\
> &\frac{c_1+c_2}{(1-(c_1c_2+r_1r_2)^2)^{3/2}}+\frac{3c_1+3c_2}{(1-(c_1c_2-\frac{r_1r_2}{3})^2)^{3/2}}
\geq 0
\end{align*}
because $c_1 \geq -c_2 $ for $(c_1,c_2) \in B$. Then $m$ is always positive in $B$, so every element $(c_1,c_2) \in C$ corresponds to the special central configuration of the form 
\begin{align*}
&\q_1=\begin{bmatrix}\sqrt{1-c_1^2} \\ 0 \\ 0 \\ c_1\end{bmatrix}, \q_2=\begin{bmatrix}-\frac{\sqrt{1-c_1^2}}{3} \\ \frac{2\sqrt{2}\sqrt{1-c_1^2}}{3} \\ 0 \\ c_1\end{bmatrix}, \q_3=\begin{bmatrix}-\frac{\sqrt{1-c_1^2}}{3}\\-\frac{\sqrt{2}\sqrt{1-c_1^2}}{3}\\ \frac{\sqrt{6}\sqrt{1-c_1^2}}{3}\\c_1\end{bmatrix}, \q_4=\begin{bmatrix} -\frac{\sqrt{1-c_1^2}}{3}\\-\frac{\sqrt{2}\sqrt{1-c_1^2}}{3}\\-\frac{\sqrt{6}\sqrt{1-c_1^2}}{3}\\c_1\end{bmatrix},\\
&\q_5=\begin{bmatrix}\sqrt{1-c_2^2} \\ 0 \\ 0 \\ c_2\end{bmatrix}, \q_6=\begin{bmatrix}-\frac{\sqrt{1-c_2^2}}{3} \\ \frac{2\sqrt{2}\sqrt{1-c_2^2}}{3} \\ 0 \\ c_2\end{bmatrix}, \q_7=\begin{bmatrix}-\frac{\sqrt{1-c_2^2}}{3}\\-\frac{\sqrt{2}\sqrt{1-c_2^2}}{3}\\ \frac{\sqrt{6}\sqrt{1-c_2^2}}{3}\\c_2\end{bmatrix}, \q_8=\begin{bmatrix} -\frac{\sqrt{1-c_2^2}}{3}\\-\frac{\sqrt{2}\sqrt{1-c_2^2}}{3}\\-\frac{\sqrt{6}\sqrt{1-c_2^2}}{3}\\c_2\end{bmatrix},\\
&m_1=m_2=m_3=m_4=1,\\
&m_5=m_6=m_7=m_8=m(c_1,c_2),\end{align*}
where $m(c_1,c_2)$ is as defined in (\ref{mtet}).
This remark completes the proof.
\end{proof}
\bigskip

To get a visual understanding of the solutions to $g$, we substitute (\ref{mtet}) into (\ref{q1tet}) and implicitly plot the solutions to the resulting equation. We can then see the set of solutions to $g(c_1,c_2)=0$. As expected, the solutions are symmetric about $c_1=-c_2$ and the right branch of the solution is a path-disconnecting set of $B$.
\smallskip

\begin{figure}[htbp] 
   \centering
   \includegraphics[width=2in]{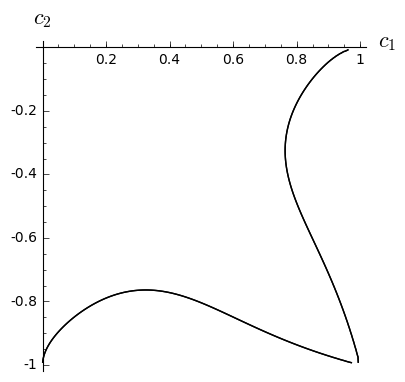}
   \caption{}
   \label{Fig1}
\end{figure}

\smallskip

\subsection{Special central configurations for four bodies in $\S^3$.}
We first show that every special central configuration of the 4-body problem in $\S^3$ occurs on a great 2-sphere, and then prove that there are no quadrilateral special central configurations on $\S^1$. Finally, we derive a necessary and sufficient condition for the existence of tetrahedral special central configurations.
\begin{proposition}
Every 4-body special central configuration in $\S^3$ occurs on a great 2-sphere.
\end{proposition}
\begin{proof}
Let $\q=(\q_1, ..., \q_4)$ be a special central configuration in $\S^3$. 
Then $\F_i=0, i=1,...,4$. Recall equation \eqref{equ:F}. We obtain 
\[ 0=\F_1=\sum_{j=2}^4 \frac{m_1m_j(\q_j-\cos d_{1j}\q_1)}{\sin^3{d_{1j}}}=\sum_{j=2}^4 \frac{m_1m_j\q_j}{\sin^3{d_{1j}}}- \sum_{j=2}^4 \frac{m_1m_j\cos d_{1j}}{\sin^3{d_{1j}}} \q_1.  \] 
This implies that the four vectors $\q_1, ...,\q_4$ are linearly dependent. Thus their rank is at most $3$, and they must lie on a great 2-sphere.
\end{proof}

\begin{proposition}There are no 4-body special central configurations on a great circle.\end{proposition}

\begin{proof}
We first derive a necessary condition on the mutual distances, and then show that no non-singular configurations satisfy the condition.
\smallskip

 We may assume that the positions of the masses are given by the polar coordinates $0=\vp_1<\vp_2<\vp_3<\vp_4<2\pi,$ with $\vp_3<\pi$ and $\vp_4 \in (\pi, \pi+\vp_2) \cup (\pi+\vp_2, \pi+\vp_3)$. The cases $\vp \in (\pi, \pi+\vp_2)$ and $\vp \in (\pi+\vp_2, \pi+\vp_3)$ differ only by a rotation $-\vp_4$ and the relabling
\begin{equation*}\q'_1=\q_4, \ \q'_2=\q_1,\ \q'_3=\q_2, \ \q'_4=\q_3,\end{equation*}
so it is sufficient to consider the case $\vp_4 \in (\pi, \pi+\vp_2)$.
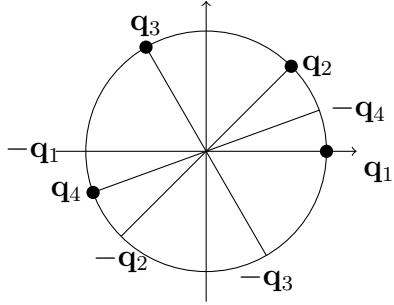
\begin{figure}[!h]
  	\centering
    \begin{tikzpicture}
    \draw (0,  0) circle (1.6);
    \draw[-] (0,0) -- (45:1.6) node[right] {$\q_2$};
    \draw[-] (0,0) -- (225:1.6) node[below] {$-\q_2$};
     \fill (45:1.6) circle (2.5pt);

   \draw[-] (0,0) -- (120:1.6) node[above] {$\q_3$};
       \draw[-] (0,0) -- (300:1.6) node[below] {$-\q_3$};
       \fill (120:1.6) circle (2.5pt);
       
    \draw[-] (0,0) -- (200:1.6) node[left] {$\q_4$};
        \draw[-] (0,0) -- (20:1.6)node[right] {$-\q_4$};
 \fill (200:1.6) circle (2.5pt);
 
    \node[below] at(0:2.3) {$\mathbf{q}_1$};
    \node at(180:2.3) {$-\mathbf{q}_1$};
   \fill (0:1.6) circle (2.5pt);
    
     \draw [->] (0, -2)--(0, 2) node (zaxis) [left]{};
           \draw [->](-2, 0)--(2, 0) node (yaxis) [right] {};
    \end{tikzpicture}
    \caption{A configuration for four masses on a great circle}
    \label{fig:4-equal-m}
    \end{figure}
Note that the potential \eqref{equ:potential} on $\S^3$ is $U=\sum\limits_{1 \leq i < j \leq N} m_i m_j \cot d_{ij}$. In this case, we can write it as
\begin{equation*}
\begin{split}
U(\vp_1,...,\vp_4)&= m_1m_2 \cot (\vp_2-\vp_1) + m_1m_3 \cot (\vp_3-\vp_1)-m_1m_4 \cot (\vp_4-\vp_1)\\
&+m_2m_3 \cot (\vp_3-\vp_2)
 +m_2m_4 \cot (\vp_4-\vp_2) +m_3m_4 \cot (\vp_4-\vp_3).
\end{split}
\end{equation*}      
Since $d_{14}=2\pi-(\vp_4-\vp_1)$, so the sign before the term $m_1m_4$ is negative. 
Since special central configuration is critical point of $U$, so taking the derivative with respect to $\vp_1$, we have 
\[ \frac{m_1m_2}{ \sin^2 (\vp_2-\vp_1)} +\frac{m_1m_3}{ \sin^2 (\vp_3-\vp_1)}-\frac{m_1m_4}{ \sin^2 (\vp_4-\vp_1)}=0.  \]
Similarly, we obtain
\begin{align} 
 \frac{m_2}{r_{12}} &+ \frac{m_3}{r_{13}}=\frac{m_4}{r_{14}}, \label{eq_scc4_s1_1}\\
 \frac{m_3}{r_{23}} &+ \frac{m_4}{r_{24}}=\frac{m_1}{r_{12}}, \label{eq_scc4_s1_2} \\
 \frac{m_1}{r_{13}} &+ \frac{m_2}{r_{23}}=\frac{m_4}{r_{34}}, \label{eq_scc4_s1_3}\\
 \frac{m_2}{r_{24}} &+ \frac{m_3}{r_{34}}=\frac{m_1}{r_{14}}, \label{eq_scc4_s1_4}  
  \end{align}
 where $r_{ij}=\sin^2d_{ij}= \sin^2 (\vp_i-\vp_j)$. Note that $\eqref{eq_scc4_s1_1}\frac{1}{r_{34}}-\eqref{eq_scc4_s1_3}\frac{1}{r_{14}}$ leads to
 \[   m_2\left(\frac{1}{r_{12}r_{34}}-\frac{1}{r_{23}r_{14}} \right) +\frac{m_3}{r_{13}r_{34}} =  \frac{m_1}{r_{13}r_{14}}. \]  
 And that $\eqref{eq_scc4_s1_4}\frac{1}{r_{13}}$ gives 
 \[  \frac{m_2}{r_{24}r_{13}}  +\frac{m_3}{r_{13}r_{34}} =  \frac{m_1}{r_{13}r_{14}}. \]
 Thus we get the necessary condition 
 \begin{equation} \label{eq_scc4_s1_5}
  \frac{1}{r_{12}r_{34}}=\frac{1}{r_{23}r_{14}}+\frac{1}{r_{13}r_{24}}.
  \end{equation}

We now show that equation \eqref{eq_scc4_s1_5} can never be satisfied. Note that $r_{ij}$ is also equal to $\sin^2 d(\pm\q_i, \pm \q_j)$. Let us look now at the upper semicircle determined by $\q_2$ and $-\q_2$. Between the two boundary points, there lie $\q_3$, $-\q_1$, and $\q_4$ consecutively. Thus 
 \[  0<d(\q_2, \q_3)< d(\q_2, -\q_1)<d(\q_2, \q_4)<\pi, \]
 and we get 
 \[ r_{12}= \sin^2 d(\q_2, -\q_1)> \min \{ \sin^2 d(\q_2, \q_3), \sin^2 d(\q_2, \q_4) \}= \min \{ r_{23}, r_{24} \}.  \]
 Similarly, by focusing on other semicircles determined by $\q_i$ and $-\q_i$, we get other similar inequalities, coming to four in total:
 \begin{align} 
 r_{12}&> \min \{ r_{13}, r_{14} \},  \ &{\rm i.e., }\ \ \  & \frac{1}{r_{12}}<\max  \left\{\frac{1}{r_{13}},\frac{1}{r_{14}}  \right\},  \label{ineq_1}\\
 r_{12}&> \min \{ r_{23}, r_{24} \},  \ &{\rm i.e., }\ \ \ & \frac{1}{r_{12}}<\max\left\{\frac{1}{r_{23}},\frac{1}{r_{24}}  \right\},  \label{ineq_2}\\
 r_{34}&> \min \{ r_{13}, r_{23} \},   \ & {\rm i.e., }\ \ \  & \frac{1}{r_{34}}<\max\left\{\frac{1}{r_{13}},\frac{1}{r_{23}}  \right\}, \label{ineq_3}\\
 r_{34}&> \min \{ r_{14}, r_{24} \},  \ & {\rm i.e., }\ \ \   & \frac{1}{r_{34}}<\max\left\{\frac{1}{r_{14}},\frac{1}{r_{24}}  \right\} . \label{ineq_4}
   \end{align}  
 Each of them leads to two different cases, so there are $16$ cases in total. However,  none of them are consistent with equation \eqref{eq_scc4_s1_5}.  We begin with (\ref{ineq_1}), so if $\frac{1}{r_{14}}\ge\frac{1}{r_{13}}$, we have $\frac{1}{r_{12}}<\frac{1}{r_{14}}$. We claim that $\frac{1}{r_{13}}\ge\frac{1}{r_{23}}$.   If not, by inequality \eqref{ineq_3}, we have $\frac{1}{r_{34}}<\frac{1}{r_{23}}$. Multiplying them we get  $ \frac{1}{r_{12}r_{34}}<\frac{1}{r_{23}r_{14}}$, which contradicts with    equation \eqref{eq_scc4_s1_5}. Thus the first inequality implies:
 \begin{itemize}
 \item if $\frac{1}{r_{14}}\ge\frac{1}{r_{13}}$, then $\frac{1}{r_{13}}\ge\frac{1}{r_{23}}$, i.e.,  $\frac{1}{r_{14}}\ge\frac{1}{r_{13}}\ge\frac{1}{r_{23}}$;
 \item if $\frac{1}{r_{13}}\ge\frac{1}{r_{14}}$, then $\frac{1}{r_{14}}\ge\frac{1}{r_{24}}$, i.e.,  $\frac{1}{r_{13}}\ge\frac{1}{r_{14}}\ge \frac{1}{r_{24}} $.
 \end{itemize}
By treating the other inequalities similarly, we get the conditions
\begin{align} 
  \frac{1}{r_{14}}&\ge\frac{1}{r_{13}}\ge\frac{1}{r_{23}},   \ \ \ {\rm or}\ \ \ \frac{1}{r_{13}}\ge\frac{1}{r_{14}}\ge \frac{1}{r_{24}}, \label{ineq_5} \\
  \frac{1}{r_{24}}&\ge\frac{1}{r_{23}}\ge\frac{1}{r_{13}},   \ \ \ {\rm or}\ \ \ \frac{1}{r_{23}}\ge\frac{1}{r_{24}}\ge \frac{1}{r_{14}}, \label{ineq_6} \\ \frac{1}{r_{13}}&\ge\frac{1}{r_{23}}\ge\frac{1}{r_{24}},   \ \ \ {\rm or} \ \ \ \frac{1}{r_{23}}\ge\frac{1}{r_{13}}\ge \frac{1}{r_{14}}, \label{ineq_7} \\ \frac{1}{r_{14}}&\ge\frac{1}{r_{24}}\ge\frac{1}{r_{23}},   \ \ \ {\rm or}\ \ \ \frac{1}{r_{24}}\ge\frac{1}{r_{14}}\ge \frac{1}{r_{13}}. \label{ineq_8}   
   \end{align}
 Denote the left inequality of condition $(i)$ by $(i+)$, and the right one by $(i-)$.  Clearly (\ref{ineq_5}+) implies (\ref{ineq_6}+) is false, and (\ref{ineq_5}$-$) implies (\ref{ineq_6}$-$) is false, so we have either  (\ref{ineq_5}+),  (\ref{ineq_6}$-$) or (\ref{ineq_5}$-$), (\ref{ineq_6}+).
\smallskip

If we assume  (\ref{ineq_5}+) and  (\ref{ineq_6}$-$), then \begin{equation*}
 \frac{1}{r_{14}}\ge\frac{1}{r_{13}}\ge\frac{1}{r_{23}}\ge\frac{1}{r_{24}}\ge \frac{1}{r_{14}},
\end{equation*}
so $r_{14}=r_{13}=r_{23}=r_{24}$. Note that $r_{13}=r_{23}$ implies $\vp_2<\frac{\pi}{2}, \vp_3=\pi-\vp_2$, since $\vp_2<\vp_3<\pi$. But $r_{13}=r_{14}$ implies $\vp_4=2\pi-\vp_3$ since $\vp_3<\pi<\vp_4<\pi+\vp_3$. Combining these two gives $\vp_4=\pi+\vp_2$, which is a singular configuration.
\smallskip

If we assume (\ref{ineq_5}$-$) and (\ref{ineq_6}+), then \begin{equation*} 
\frac{1}{r_{13}}\ge\frac{1}{r_{14}}\ge \frac{1}{r_{24}}\ge\frac{1}{r_{23}}\ge\frac{1}{r_{13}},
\end{equation*}
which, similarly to the previous case, leads to a singular configuration.
\smallskip

Therefore there is no nonsingular configuration that satisfies (\ref{ineq_5})-(\ref{ineq_8}), so there are no special central configurations with four masses on a great circle.
\end{proof}

To prove our next proposition, we will rely on the following linear algebra result.

\begin{lemma}
Let $\{\v_0, \v_1, \ldots, \v_n\}$ be a collection of vectors in $\R^n$ with rank $n$. Then $$D_0\v_0-D_1\v_1+\ldots+(-1)^nD_n\v_n=0,$$
where $D_k=\det(\v_0,\ldots, \v_{k-1},\v_{k+1}, \ldots, \v_n)$.\end{lemma}
\begin{proof}
Without loss of generality, we may assume that $D_0 \neq 0$. Then we can use Cramer's rule to solve the linear system $$(\v_1,\ldots, \v_n)\mathbf{x}=-\v_0, \mathbf{x}=(x_1, x_2, \ldots, x_n)^T.$$ We get \begin{align*}
x_k&=\frac{\det(\v_1, \ldots, \v_{k-1}, -\v_0, \v_{k+1}, \ldots, \v_n)}{\det(\v_1, \ldots, \v_n)}\\
&=(-1)^k\frac{\det(\v_0, \v_1, \ldots, \v_{k-1}, \v_{k+1}, \ldots, \v_n)}{D_0}=(-1)^k\frac{D_k}{D_0}.\end{align*}
Then $\sum_{k=0}^n (-1)^k D_k\v_k=0$, so the proof is complete.
\end{proof}

\begin{proposition} Let $\q$ be a tetrahedron configuration of four masses $m_0, m_1, m_2, m_3,$ on $\S^2_{xyz}$ of the form
$$\q_0=\begin{bmatrix}1\\0\\0\end{bmatrix}, \q_1=\begin{bmatrix}x_1\\y_1\\0\end{bmatrix}, \q_2=\begin{bmatrix}x_2\\y_2\\z_2\end{bmatrix}, \q_3=\begin{bmatrix}x_3\\y_3\\z_3\end{bmatrix}.$$
Then $\q$ is a special central configuration if and only if the following three conditions are satisfied:
\begin{enumerate}
\item $\q_0, \q_1, \q_2, \q_3$ are not all in the same hemisphere;
\item $\sin{d_{01}}\sin{d_{23}}=\sin{d_{02}}\sin{d_{13}}=\sin{d_{03}}\sin{d_{12}}$;
\item $m_0=-m_3\frac{D_0\sin^3 {d_{01}}}{D_3\sin^3 {d_{13}}}, m_1=m_3\frac{D_1\sin^3{d_{01}}}{D_3\sin^3{d_{03}}}$, and $m_2=-m_3\frac{D_2\sin^3{d_{02}}}{D_3\sin^3{d_{03}}}$, where $D_0=\det(\q_1, \q_2, \q_3), D_1=\det(\q_0, \q_2, \q_3), D_2=\det(\q_0, \q_1, \q_3)$, and $D_3=\det(\q_0, \q_1, \q_2)$.
\end{enumerate}
\end{proposition}
\begin{proof}
Suppose $\q$ is a special central configuration. Then clearly the four masses are not all in one hemisphere, and $\F_i=\nabla_{\q_i} U=0$ for $i=0,1,2,3$. Consider the $z$ components of $\F_0$ and $\F_1$, $$\frac{m_2z_2}{\sin^3{d_{02}}}+\frac{m_3z_3}{\sin^3{d_{03}}}=0,\ \ \ \frac{m_2z_2}{\sin^3{d_{12}}}+\frac{m_3z_3}{\sin^3{d_{13}}}=0.$$
Since there are no special central configurations on a great circle, $y_1$, $z_2$, and $z_3$ are non-zero, so $\sin^3{d_{03}}\sin^3{d_{12}}=\sin^3{d_{02}}\sin^3{d_{13}}$. By symmetry and a relabling of the masses, we also get the relation $\sin^3{d_{01}}\sin^3{d_{23}}=\sin^3{d_{03}}\sin^3{d_{12}}$. Therefore $$\sin{d_{03}}\sin{d_{12}}=\sin{d_{02}}\sin{d_{13}}=\sin{d_{01}}\sin{d_{23}}.$$
For the masses, we look at $$\F_0=m_1\frac{\q_1-\cos{d_{01}}\q_0}{\sin^3{d_{01}}}+m_2\frac{\q_2-\cos{d_{02}}\q_0}{\sin^3{d_{02}}}+m_3\frac{\q_3-\cos{d_{03}}\q_0}{\sin^3{d_{03}}},$$
we have the z component $\frac{m_2z_2}{\sin^3{d_{02}}}+\frac{m_3z_3}{\sin^3{d_{03}}}=0$, which implies \begin{equation}\label{tm2}m_2=-m_3\frac{\sin^3{d_{02}}z_3}{\sin^3{d_{03}}z_2}=-m_3 \frac{\sin^3{d_{02}}y_1z_3}{\sin^3{d_{03}}y_1z_2}=-m_3\frac{D_2\sin^3{d_{02}}}{D_3\sin^3{d_{03}}}.\end{equation}
The $y$ component is $\frac{m_1y_1}{\sin^3{d_{01}}}+\frac{m_2y_2}{\sin^3{d_{02}}}+\frac{m_3y_3}{\sin^3{d_{03}}}=0$, which, after substituting (\ref{tm2}), gives \begin{equation}\label{tm1}
m_1=m_3\frac{(y_2z_3-z_2y_3)\sin^3{d_{01}}}{y_1z_2\sin^3{d_{03}}}=m_3\frac{D_1\sin^3{d_{01}}}{D_3\sin^3{d_{03}}}.\end{equation}
For $m_0$, we look at the inner product of $(y_1,-x_1, 0)^T$ with $$\F_1=m_0\frac{\q_0-\cos{d_{01}}\q_1}{\sin^3{d_{01}}}+m_2\frac{\q_2-\cos{d_{12}}\q_1}{\sin^3{d_{12}}}+m_3\frac{\q_3-\cos{d_{13}}\q_1}{\sin^3{d_{13}}}=\mathbf{0}$$
to get \begin{align*}
0&=\frac{m_0y_1}{\sin^3{d_{01}}}+\frac{m_2(x_2y_1-x_1y_2)}{\sin^3{d_{12}}}+\frac{m_3(y_1x_3-x_1y_3)}{\sin^3{d_{13}}}\\
&=\frac{m_0y_1}{\sin^3{d_{01}}}-\frac{m_3\sin^3{d_{02}}}{\sin^3{d_{03}}\sin^3{d_{12}}} \frac{z_3(x_2y_1-x_1y_2)}{z_2}+\frac{m_3(x_3y_1-x_1y_3)}{\sin^3{d_{13}}}\\
&=\frac{m_0y_1}{\sin^3{d_{01}}}-\frac{m_3(x_2y_1z_3-x_1y_2z_3-y_1x_3z_2+x_1y_3z_2)}{z_2\sin^3{d_{13}}}.\end{align*}
So we have \begin{equation}
\label{tm0}m_0=m_3\frac{\sin^3{d_{01}}}{\sin^3{d_{13}}}\frac{y_1(x_2z_3-x_3z_2)-x_1(y_2z_3-y_3z_2)}{y_1z_2}=-m_3\frac{D_0\sin^3{d_{01}}}{D_3\sin^3{d_{13}}}.\end{equation}
\medskip

Conversely, Suppose that $\q$ is a configuration satisfying the above three conditions. We prove that $\q$ is a special central configuration, i.e. $F_i=\mathbf{0}$ and $m_i>0$ for $i=0,1,2,3$. We can easily see that $\F_i \cdot \q_i=0$ for $i=0,1,2,3$. $\F_0\cdot e_1$ is clearly zero since $\q_1=e_1$, and $\F_0\cdot e_2=\F_0 \cdot e_3=0$ by (\ref{tm2}) and (\ref{tm1}), so $\F_0=\mathbf{0}$.
\smallskip

For $i=1,2,3$, $\F_i=\mathbf{0}$ if and only if $\F_i \cdot \v_{ij}=0, j=1,2,3$, where $\v_{i1}=\q_i$, and $\{\v_{i1}, \v_{i2}, \v_{i3}\}$ is an orthonormal basis of $\R^3$. As shown above, $\F_i \cdot \v_{i1}=0$. For $i=1, j=2,3$, we have \begin{align*}
\F_1 \cdot \v_{1j}&=\biggl(m_0\frac{\q_0-\cos{d_{01}}\q_1}{\sin^3{d_{01}}}+m_2\frac{\q_2-\cos{d_{12}}\q_1}{\sin^3{d_{12}}}+m_3\frac{\q_3-\cos{d_{13}}\q_1}{\sin^3{d_{13}}}\biggr) \cdot \v_{1j}\\
&=\biggl(\frac{m_0\q_0}{\sin^3{d_{01}}}+\frac{m_2\q_2}{\sin^3{d_{12}}}+\frac{m_3\q_3}{\sin^3{d_{13}}}\biggr)\cdot \v_{1j}\\
&= \biggl(-m_3\frac{D_0\q_0}{D_3\sin^3{d_{13}}}-m_3\frac{D_2\sin^3{d_{02}}\q_2}{D_3\sin^3{d_{03}}\sin^3{d_{12}}}+\frac{m_3\q_3}{\sin^3{d_{13}}}\biggr) \cdot \v_{1j}\\
&=-\frac{m_3}{D_3\sin^3{d_{13}}}(D_0\q_0+D_2\q_2-D_3\q_3)\cdot \v_{1j}
=-\frac{m_3}{D_3\sin^3{d_{13}}} D_1\q_1\cdot v_{1j}=0,
\end{align*}
the second last equality following by the previous lemma.

Through similar computations, we can see that for $j=2,3$,
\begin{align*}
\F_2\cdot \v_{2j}&=-\frac{m_3}{D_3\sin^3{d_{23}}}(D_0\q_0-D_1\q_1-D_3\q_3)\cdot \v_{2j}=\frac{m_3D_2}{D_3\sin^3{d_{23}}}\q_2 \cdot \v_{2j}=0,\\
\F_3 \cdot \v_{3j}&=-\frac{m_3\sin^3{d_{01}}}{D_3\sin^3{d_{13}}\sin^3{d_{03}}}(D_0\q_0-D_1\q_1+D_2\q_2)\cdot \v_{3j}\\
&=-\frac{m_3\sin^3{d_{01}}}{\sin^3{d_{13}}\sin^3{d_{03}}}\q_3\cdot \v_{3j}=0.\end{align*}
Therefore $\F_i=\mathbf{0}$ for $i=0,1,2,3$.
To show that the masses are positive, we first show that $D_i \neq 0$, for $i=0,1,2,3$. If not, then three of the masses lie in a great circle of $\S^2_{xyz}$, so the four masses all lie in one hemisphere. 
\smallskip

Without loss of generality, assume $D_3>0$. Consider the two-dimensional subspace $V_{12}$ determined by $\q_1, \q_2$. Since the configuration is not in one hemisphere, $V_{12}$ must separate $\q_0$ and $\q_3$. Then $$D_3=\det(\q_0, \q_1, \q_2)=\det(\q_1, \q_2 ,\q_0)>0\text{ implies }D_0=\det(\q_1,\q_2,\q_3)<0.$$
Similarly, the subspace $V_{02}$ separates $\q_1$ and $\q_3$, so $$\det(\q_0, \q_2, \q_1)=-\det(\q_0, \q_1, \q_2)=-D_3<0\text{ implies }D_1=\det(\q_0, \q_2, \q_3)>0,$$
and the subspace $V_{01}$ separates $\q_2$ and $\q_3$, so $$\det(\q_0, \q_1 \q_2)=D_3>0\text{ implies } D_2=\det(\q_0, \q_1, \q_3) <0.$$
Then $m_0>0, m_1>0, m_2>0$ if and only if $m_3>0$, so $\q$ is a special central configuration.
\end{proof}

\subsection{Special central configurations for five bodies in $\S^3$}

In this section we generalize the method from the previous proof from tetrahedra in $\S^2_{xyz}$ to pentatopes in $\S^3$ to prove the following result.
\begin{proposition} Let $\q$ be a pentatope configuration for five masses, $m_0, m_1, m_2$, $m_3, m_4,$ in $\S^3$ of the form 
$$\q_0=\begin{bmatrix}1\\0\\0\\0\end{bmatrix}, \q_1=\begin{bmatrix}x_1\\y_y\\0\\0\end{bmatrix}, \q_2=\begin{bmatrix}x_2\\y_2\\z_2\\0\end{bmatrix}, \q_3=\begin{bmatrix}x_3\\y_3\\z_3\\w_3\end{bmatrix}, \q_4=\begin{bmatrix}x_4\\y_4\\z_4\\w_4\end{bmatrix}.$$ 
Then $\q$ is a special central configuration if and only if the following conditions are satisfied:
\begin{enumerate}
\item $\q_0, \q_1, \q_2, \q_3, \q_4$ are not all in one hemisphere;
\item $\frac{\sin{d_{01}}}{\sin{d_{04}}}=\frac{\sin{d_{12}}}{\sin{d_{24}}}=\frac{\sin{d_{13}}}{\sin{d_{34}}}$;
\item $\frac{\sin{d_{02}}}{\sin{d_{04}}}=\frac{\sin{d_{12}}}{\sin{d_{14}}}=\frac{\sin{d_{23}}}{\sin{d_{34}}}$;
\item $\frac{\sin{d_{03}}}{\sin{d_{04}}}=\frac{\sin{d_{13}}}{\sin{d_{14}}}=\frac{\sin{d_{23}}}{\sin{d_{24}}}$;
\item $m_0=m_4\frac{D_0\sin^3{d_{01}}}{D_4\sin^3{d_{14}}}$, $m_1=-m_4\frac{D_1\sin^3{d_{01}}}{D_4\sin^3{d_{04}}}$, $m_2=m_4\frac{D_2\sin^3{d_{02}}}{D_4\sin^3{d_{04}}}$, and $m_3=-m_4\frac{D_3\sin^3{d_{03}}}{D_4\sin^3{d_{04}}}$, where $$D_0=\det(\q_1,\q_2,\q_3,\q_4), D_1=\det(\q_0,\q_2,\q_3,\q_4), D_2=\det(\q_0, \q_1, \q_3, \q_4),$$  $$D_3=\det(\q_0, \q_1, \q_2, \q_4),\text{ and }D_4=\det(\q_0, \q_1, \q_2, \q_3).$$
\end{enumerate}
\end{proposition}
\begin{proof}
Suppose $\q$ is a special central configuration. Then clearly the five masses are not all in one hemisphere, and we have $\F_i=0$ for $i=0,1,2,3,4$. Consider the $w$ components of $\F_0, \F_1$, and $\F_2$, $$\frac{m_3w_3}{\sin^3{d_{03}}}+\frac{m_4w_4}{\sin^3{d_{04}}}=0, \frac{m_3w_3}{\sin^3{d_{13}}}+\frac{m_4w_4}{\sin^3{d_{14}}}=0, \frac{m_3w_3}{\sin^3{d_{23}}}+\frac{m_4w_4}{\sin^3{d_{34}}}=0.$$ Since we are assuming $\q$ does not lie on a great sphere, $y_1, z_2, w_3, w_4$ are non-zero, so \begin{align*}\frac{\sin{d_{03}}}{\sin{d_{04}}}=\frac{\sin{d_{13}}}{\sin{d_{14}}}=\frac{\sin{d_{23}}}{\sin{d_{24}}}.\\
\intertext{By symmetry and a relabling of the masses, we also obtain the relations}
\frac{\sin{d_{01}}}{\sin{d_{04}}}=\frac{\sin{d_{12}}}{\sin{d_{24}}}=\frac{\sin{d_{13}}}{\sin{d_{34}}}\\
\intertext{and}
\frac{\sin{d_{02}}}{\sin{d_{04}}}=\frac{\sin{d_{12}}}{\sin{d_{14}}}=\frac{\sin{d_{23}}}{\sin{d_{34}}}.
\end{align*} 
If we look at 
$$
\F_0=m_1\frac{\q_1-\cos{d_{01}}\q_0}{\sin^3{d_{01}}}+m_2\frac{\q_2-\cos{d_{02}}\q_0}{\sin^3{d_{02}}}+m_3\frac{\q_3-\cos{d_{03}}\q_0}{\sin^3{d_{03}}}+m_4\frac{\q_4-\cos{d_{04}}\q_0}{\sin^3{d_{04}}},
$$
which is $\mathbf{0}$,
we see that the $w$ component $\frac{m_3w_3}{\sin^3{d_{03}}}+\frac{m_4w_4}{\sin^3{d_{04}}}=0$ gives 
\begin{equation}
\label{pm3} m_3=-m_4\frac{w_4\sin^3{d_{03}}}{w_3\sin^3{d_{04}}}=-m_4\frac{y_1z_2w_4\sin^3{d_{03}}}{y_1z_2w_3\sin^3{d_{04}}}=-m_4\frac{D_3\sin^3{d_{03}}}{D_4\sin^3{d_{04}}}.\end{equation}
After substituting (\ref{pm3}) into the $z$ component $\frac{m_2z_2}{\sin^3{d_{02}}}+\frac{m_3z_3}{\sin^3{d_{03}}}+\frac{m_4z_4}{\sin^3{d_{04}}}=0$, we have \begin{equation}
\label{pm2} m_2=m_4\frac{y_1(z_3w_4-z_4w_3)\sin^3{d_{02}}}{y_1z_2w_3\sin^3{d_{04}}}=m_4\frac{D_2\sin^3{d_{02}}}{D_4\sin^3{d_{04}}}.\end{equation}
After substituting (\ref{pm3}) and (\ref{pm2}) into the $y$ component $\frac{m_1y_1}{\sin^3{d_{01}}}+\frac{m_2y_2}{\sin^3{d_{02}}}+\frac{m_3y_3}{\sin^3{d_{03}}}+\frac{m_4y_4}{\sin^3{d_{04}}}=0$, we obtain\begin{equation}
\label{pm1} m_1=-m_4\frac{y_2(z_3w_4-z_4w_3)-z_2(y-3w_4-y_4w_3)}{y_1z_2w_3}\frac{\sin^3{d_{01}}}{\sin^3{d_{04}}}=-m_4\frac{D_1\sin^3{d_{01}}}{D_4\sin^3{d_{04}}}.\end{equation}
We obtain $m_0$ by taking the inner product of $(y_1,-x_1,0,0)^T$ with \begin{align*}\F_1=&m_0\frac{\q_0-\cos{d_{01}}\q_1}{\sin^3{d_{01}}}+m_2\frac{\q_2-\cos{d_{12}}\q_1}{\sin^3{d_{12}}}\\
&+m_3\frac{\q_3-\cos{d_{13}}\q_1}{\sin^3{d_{13}}}+m_4\frac{\q_4-\cos{d_{14}}\q_1}{\sin^3{d_{14}}}=\mathbf{0}\end{align*}
to get \begin{align*}
0=&\frac{m_0y_1}{\sin^3{d_{01}}}+m_2\frac{x_2y_1-x_1y_2}{\sin^3{d_{12}}}+m_3\frac{x_3y_1-x_1y_3}{\sin^3{d_{13}}}+m_4\frac{x_4y_1-x_1y_4}{\sin^3{d_{14}}}\\
=&\frac{m_0y_1}{\sin^3{d_{01}}}+m_4\frac{(x_2y_1-x_1y_2)(z_3w_4-z_4w_3)\sin^3{d_{02}}}{z_2w_3\sin^3{d_{12}}\sin^3{d_{04}}}\\
&-m_4\frac{(x_3y_1-x_1y_3)z_2w_4\sin^3{d_{03}}}{z_2w_3\sin^3{d_{13}}\sin^3{d_{04}}}+m_4\frac{(x_4y_1-x_1y_4)z_2w_3}{z_2w_3\sin^3{d_{14}}}\\
=&\frac{m_0y_1}{\sin^3{d_{01}}}-m_4\frac{(x_1y_2-x_2y_1)(z_3w_4-z_4w_3)+(x_1y_4-x_4y_1)(w_3-w_4)z_2}{\sin^3{d_{14}}}.\end{align*}
Then we have \begin{align}
\label{pm0}m_0&=m_4\frac{(x_1y_2-x_2y_1)(z_3w_4-z_4w_3)+(x_1y_3-x_3y_1)(w_3-w_4)z_2}{y_1z_2w_3}\frac{\sin^3{d_{01}}}{\sin^3{d_{14}}}\\
\nonumber&=m_4\frac{D_0\sin^3{d_{01}}}{D_4\sin^3{d_{14}}}.\end{align}
\medskip

Conversely, Suppose that $\q$ is a configuration which satisfies the above 5 conditions. We now prove that $\q$ is a special central configuration, i.e., $\F_i=\mathbf{0}$, $m_i>0$ for $i=0,1,2,3,4$. $\F_i \cdot \q_i$ is clearly zero, and $\F_0$ is zero since $e_1=\q_0$ and $\F_0\cdot e_2=\F_0\cdot e_3=\F_0\cdot e_4=0$ by (\ref{pm3})-(\ref{pm1}).
\smallskip

For $i=1,2,3,4$, $\F_i=0$ if and only if $\F_i \cdot \v_{ij}=0$, $j=1,2,3,4$, where $\v_{i1}=\q_i$, and $\{\v_{i1}, \v_{i2}, \v_{i3}, \v_{i4}\}$ form an orthonormal basis of $\R^3$. As shown above, $\F_i \cdot \v_{i1}=0$. For $i=1, j=2,3,4$, we have \begin{align*}
\F_1 \cdot \v_{1j}= &\biggl(m_0\frac{\q_0-\cos{d_{01}}\q_1}{\sin^3{d_{01}}}+m_2\frac{\q_2-\cos{d_{12}}\q_1}{\sin^3{d_{12}}}\\
&+m_3\frac{\q_3-\cos{d_{13}}\q_1}{\sin^3{d_{13}}}+m_4\frac{\q_4-\cos{d_{14}}\q_1}{\sin^3{d_{14}}}\biggr)\cdot  \v_{1j}\\
=&\biggl(\frac{m_0\q_0}{\sin^3{d_{01}}}+\frac{m_2\q_2}{\sin^3{d_{12}}}+\frac{m_3\q_3}{\sin^3{d_{13}}}+\frac{m_4\q_4}{\sin^3{d_{14}}}\biggr)\cdot \v_{1j}\\
=&\biggl(m_4\frac{D_0\q_0}{D_4\sin^3{d_{14}}}+m_4\frac{D_2\sin^3{d_{02}}\q_2}{D_4\sin^3{d_{04}}\sin^3{d_{12}}}\\
&-m_4\frac{D_3\sin^3{d_{03}}\q_3}{D_4\sin^3{d_{04}}\sin^3{d_{13}}}+\frac{m_4\q_4}{\sin^3{d_{14}}}\biggr)\cdot \v_{1j}\\
=&\frac{m_4}{D_4\sin^3{d_{14}}}(D_0\q_0+D_2\q_2-D_3\q_3+D_4\q_4)\cdot \v_{1j}\\
=&\frac{m_4D_1}{D_4\sin^3{d_{14}}}\q_1\cdot \v_{1j}=0.
\end{align*}
Through similar computations, we see that for $j=2,3,4$, \begin{align*}
\F_2\cdot \v_{2j}&=\frac{m_4}{D_4\sin^3{d_{24}}}(D_0\q_0-D_1\q_1-D_3\q_3+D_4\q_4)\cdot \v_{2j}=\frac{-m_4D_2}{D_4\sin^3{d_{24}}}\q_2\cdot \v_{2j}=0,\\
\F_3\cdot \v_{3j}&=\frac{m_4}{D_4\sin^3{d_{34}}}(D_0\q_0-D_1\q_1+D_2\q_2+D_4\q_4)\cdot \v_{3j}=\frac{m_4D_3}{D_4\sin^3{d_{23}}}\q_3\cdot \v_{3j}=0,\\
\intertext{and}
\F_4 \cdot \v_{4j}&=\frac{m_4\sin^3{d_{01}}}{D_4\sin^3{d_{14}}\sin^3{d_{04}}}(D_0\q_0-D_1\q_1+D_2\q_2-D_3\q_3)\cdot \v_{4j}\\
&=\frac{-m_4\sin^3{d_{01}}}{\sin^3{d_{14}}\sin^3{d_{04}}}\q_3\cdot \v_{4j}=0.\end{align*}
Therefore $\F_i=0$ for $i=0,1,2,3,4$. To show that the masses are positive, we first show that $D_i \neq 0$ for $i=0,1,2,3,4$. If not, then four of the masses lie in a great sphere, so the five masses all lie in one hemisphere.
\smallskip

Without loss of generality, assume $D_4>0$. Consider the 3-dimensional subspace $V_{123}$. Since the configuration is not in one hemisphere, $V_{123}$ must separate $\q_0$ and $\q_4$. Then 
$$
\det(\q_1, \q_2, \q_3, \q_0)=-\det(\q_0, \q_1, \q_2, \q_3)=-D_4<0
$$
implies that 
$$
D_0=\det(\q_1, \q_2, \q_3, \q_4)>0.
$$
Similarly, the subspace $V_{023}$ separates $\q_1$ and $\q_4$, so 
$$
\det(\q_0, \q_2, \q_3, \q_1)=\det(\q_0, \q_1, \q_2, \q_3)=D_4>0
$$
implies that
$$
D_1=\det(\q_0, \q_2, \q_3, \q_4)<0;
$$
The subspace $V_{013}$ separates $\q_2$ and $\q_4$, so 
$$
\det(\q_0, \q_1, \q_3, \q_2)=-\det(\q_0, \q_1, \q_2, \q_3)=-D_4<0
$$
implies that
$$
D_2=\det(\q_0, \q_1, \q_3, \q_4)>0;
$$
and the subspace $V_{012}$ separates $\q_3$ and $\q_4$, so 
$$
\det(\q_0, \q_1, \q_2, \q_3)=D_4>0
$$
implies that
$$
D_3=\det(\q_0, \q_1, \q_2, \q_4)<0.
$$
Then $m_0>0, m_1>0, m_2>0, m_3>0$ if and only if $m_4>0$, so $\q$ is a special central configuration. This remark completes the proof.
\end{proof}

\bigskip

\noindent{\bf Acknowledgments.} This research was supported in part by an USRA Fellowship from NSERC of Canada (for Eric Boulter), a Discovery Grant from the same institution (for Florin Diacu), as well as a University of Victoria Scholarship and a Geoffrey Fox Graduate Fellowship (for Shuqiang Zhu).



\begin{thebibliography}{99}

\bibitem{Ben} S.~Benenti, \emph{Analytical Cosmology---An Axiomatic Setting of the Isotropic Cosmological Models}, http://www.sergiobenenti.it/In/AnCosm-March-14.pdf

\bibitem{Bolyai} W.~Bolyai and J.~Bolyai, {\it Geometrische Untersuchungen}, Teubner, Leipzig-Berlin, 1913.

\bibitem{Diacu} F. \ Diacu, {\it Relative equilibria of the curved $N$-body problem}, Atlantis Studies in Dynamical Systems, vol. 1, Atlantis Press, Amsterdam, 2012.

\bibitem{Diacu2} F. \ Diacu, The classical $N$-body problem in the context of curved space, {\it Canadian J.\ Math.} (to appear).

\bibitem{Diacu3} F. \ Diacu, Relative equilibria in the 3-dimensional curved $N$-body problem, {\it Memoirs Amer. Math. Soc.} {\bf 228}, 1071 (2013).

\bibitem{Diacu5} F.\ Diacu, On the singularities of the curved $N$-body problem, {\it Trans.\ Amer.\ Math.\ Soc.} {\bf 363}, 4 (2011), 2249--2264.

\bibitem{Diacu51} F.\ Diacu, S.\ Ibrahim, and J.\ Sniatycki, The continuous transition of Hamiltonian vector fields through manifolds of constant curvature, {\it J.\ Math.\ Phys.}  {\bf 57}, 062701 (2016); http://dx.doi.org/10.1063/1.4953371.

\bibitem{DiacuKordlou} F. \ Diacu and S. \ Kordlou, Rotopulsators of the curved $N$-body problem, {\it J. Differential Equations} {\bf 255} (2013) 2709-2750.

\bibitem{Diacu6} F.~Diacu, E.~P\'erez-Chavela, and M.~Santoprete, The
$N$-body problem in spaces of constant curvature. Part I: Relative equilibria,
{\it J.\ Nonlinear Sci.} {\bf 22}, 2 (2012), 247--266, DOI: 10.1007/s00332-011-9116-z.

\bibitem{Diacu7} F.~Diacu, E.~P\'erez-Chavela, and M.~Santoprete, The $N$-body problem in spaces of constant curvature. Part II: Singularities,
{\it J.\ Nonlinear Sci.} {\bf 22}, 2 (2012), 267--275, DOI: 10.1007/s00332-011-9117-y.


\bibitem{DiacuPChavelaGRVictoria} F.\ Diacu, E.\ P\'erez-Chavela, and J.\ Guadalupe Reyes-Victoria, An intrinsic approach in the curved $n$-body problem. The negative curvature case, {\it J. Differential Equations} {\bf 252} (2012), 4529-4562.

\bibitem{DiacuSanchez} F.\ Diacu, J.M.\ Sanchez-Cerritos, and S.\ Zhu, Stability of fixed points and associated relative equilibria of the 3-body problem on $\S^1$ and $\S^2$, {\it J.\ Dynam.\ Differential Equations} (to appear).

\bibitem{DiacuZhu} F. \ Diacu, C. \ Stoica, and S. \ Zhu, Central configurations of the curved $N$-body problem, arXiv:1603.03342.

\bibitem{DiacuThorn} F. \ Diacu and B. \ Thorn, Rectangular orbits of the curved 4-body problem, {\it Proc. Amer. Math. Soc.} {\bf143} (2015), 1583-1593.

\bibitem{Garcia} L.C.\ Garc\'ia-Naranjo, J.C.\ Marrero, E.\ P\'erez-Chavela, M.\ Rodr\'iguez-Olmos, Classification and stability of relative equilibria for the two-body problem in the hyperbolic space of dimension 2, arXiv:1505.01452.

\bibitem{Hubble} E. \ Hubble, A relation between distance and radial velocity among extra-galactic nebulae, {\it Proc.\ National Acad.\ Sci.}, {\bf 15}, 3 (1929), 168-73.

\bibitem{Lobachevsky} N.~I.~Lobachevsky, The new foundations of geometry with full theory of parallels [in Russian], 1835-1838, in Collected Works, vol.\ 2, GITTL, Moscow, 1949.

\bibitem{Martinez1} R.\ Mart\'inez and C.\ Sim\'o, On the stability of the Lagrangian homographic solutions in a curved three-body problem on $\mathbb S^2$, {\it Discrete Contin.\ Dyn.\ Syst.\ Ser.\ A} {\bf 33} (2013) 1157--1175.

\bibitem{PChavelaGRVictoria} E. \ P\'erez-Chavela and J. \ Guadalupe Reyes-Victoria, An intrinsic approach in the curved $n$-body problem. The positive curvature case, {\it Trans. Amer. Math. Soc.} {\bf 364}, 7 (2012), 3805-3827.

\bibitem{Shchepetilov} A.V.~Shchepetilov,  Nonintegrability of the two-body problem in constant curvature spaces, {\it J.\ Phys. A: Math.\ Gen.} V.\ 39 (2006), 5787-5806; corrected version at math.DS/0601382.

\bibitem{Tibboel1} P.\ Tibboel, Polygonal homographic orbits in spaces of constant curvature, {\it Proc.\ Amer.\ Math.\ Soc.} {\bf 141} (2013), 1465--1471.

\bibitem{Tibboel2} P.\ Tibboel, Existence of a class of rotopulsators, {\it J.\ Math.\ Anal.\ Appl.} {\bf 404} (2013), 185--191.

\bibitem{Tibboel3} P.\ Tibboel, Existence of a lower bound for the distance between point masses of relative equilibria in spaces of constant curvature, {\it J.\ Math.\ Anal.\ Appl.} {\bf 416} (2014), 205--211.

\bibitem{Zhu} S.\ Zhu, Eulerian relative equilibria of the curved 3-body problems in $\mathbb S^2$, {\it Proc.\ Amer.\ Math.\ Soc.} {\bf  142} (2014), 2837--2848.
\end{thebibliography}
\end{document}